\documentclass[11pt]{amsart}

\usepackage{amsmath, amssymb, epsfig}
\usepackage{amsfonts}
\usepackage{mathrsfs}
\usepackage[arrow,matrix,curve,cmtip,ps]{xy}
\usepackage{xcolor}

\usepackage{amsthm}

\allowdisplaybreaks

\newtheorem{theorem}{Theorem}[section]
\newtheorem{lemma}[theorem]{Lemma}
\newtheorem{proposition}[theorem]{Proposition}

\newtheorem{corollary}[theorem]{Corollary}

\newtheorem*{theorem*}{Theorem}
\theoremstyle{remark}
\newtheorem{remark}[theorem]{Remark}
\newtheorem{definition}[theorem]{Definition}
\newtheorem{example}[theorem]{Example}

\numberwithin{equation}{section}

\newcommand{\Z}{\mathbb{Z}}
\newcommand{\Q}{\mathbb{Q}}
\newcommand{\N}{\mathbb{N}}
\newcommand{\R}{\mathbb{R}}
\newcommand{\C}{\mathbb{C}}

\newcommand{\F}{\mathbb{F}}
\newcommand{\im}{\operatorname{im }}
\newcommand{\coker}{\operatorname{coker }}
\newcommand{\Ab}{\textnormal{\textbf{Ab}}}

\newcommand{\rank}{\operatorname{rank}}
\newcommand{\reg}{\textnormal{reg}}
\newcommand{\sing}{\textnormal{sing}}
\newcommand{\corank}{\operatorname{corank}}
\newcommand{\cha}{\operatorname{char}}
\newcommand{\diag}{\operatorname{diag}}
\newcommand{\bigslant}[2] {{\left.\raisebox{.2em}{$#1$} \middle/ \raisebox{-.2em}{$#2$}\right.}}
\newcommand{\modulo}{\operatorname{mod}}

\newcommand{\sign}{\operatorname{sign}}

\newcommand{\gae}{\lower 2pt \hbox{$\, \buildrel {\scriptstyle >}\over {\scriptstyle
\sim}\,$}}

\newcommand{\lae}{\lower 2pt \hbox{$\, \buildrel {\scriptstyle <}\over {\scriptstyle
\sim}\,$}}

\newcommand{\MU}[1]{
\setbox0\hbox{$#1$}
\setbox1\hbox{$W$}
\ifdim\wd0>\wd1 #1^{\sim} \else \widetilde{#1} \fi
}

\begin{document}
\title[$K$-theory for Leavitt path algebras]{$\boldsymbol{K}$-theory for Leavitt path algebras: computation and classification}

\author{James Gabe, Efren Ruiz, Mark Tomforde, and Tristan Whalen}

\address{Department of Mathematical Sciences, University of Copenhagen, Universitetsparken 5,
DK-2100 Copenhagen, Denmark}
\email{gabe@math.ku.dk}
\address{Department of Mathematics, University of Hawaii, Hilo, 200 W. Kawili St., Hilo, Hawaii, 96720-4091 USA}
\email{ruize@hawaii.edu }
\address{Department of Mathematics \\ University of Houston \\ Houston, TX 77204-3008 \\USA}
\email{tomforde@math.uh.edu}
\address{Department of Mathematics \\ University of Houston \\ Houston, TX 77204-3008 \\USA}
\email{tgwhalen@math.uh.edu}

\thanks{This work was supported by the Danish National Research Foundation through the Centre for Symmetry and Deformation (DNRF92).  The second author was supported by a grant from the Simons Foundation (\#279369 to Efren Ruiz). The third author was supported by a grant from the Simons Foundation (\#210035 to Mark Tomforde).}

\date{\today}



\begin{abstract}
We show that the long exact sequence for $K$-groups of Leavitt path algebras deduced by Ara, Brustenga, and Corti\~nas extends to Leavitt path algebras of countable graphs with infinite emitters in the obvious way.  Using this long exact sequence, we compute explicit formulas for the higher algebraic $K$-groups of Leavitt path algebras over certain fields, including all finite fields and all algebraically closed fields.  We also examine classification of Leavitt path algebras using $K$-theory.  
It is known that the $K_0$-group and $K_1$-group do not suffice to classify purely infinite simple unital Leavitt path algebras of infinite graphs up to Morita equivalence when the underlying field is the rational numbers.  We prove for these Leavitt path algebras, if the underlying field is a number field (which includes the case when the field is the rational numbers), then the pair consisting of the $K_0$-group and the $K_6$-group does suffice to  classify these Leavitt path algebras up to Morita equivalence.
\end{abstract}

\maketitle

\section{Introduction}

A major program in the study of $C^*$-algebras is the classification of $C^*$-algebras using invariants provided by (topological) $K$-theory.  The strongest classification results have been obtained in the case of simple $C^*$-algebras, and one of the preeminent results in this setting is the Kirchberg-Phillips classification theorem (see \cite{Kir-pre} and \cite[Theorem~4.2.4]{Phi}), which states that under mild hypotheses a purely infinite simple $C^*$-algebra is classified up to Morita equivalence by the pair consisting of the $K_0$-group and $K_1$-group of the $C^*$-algebra.

It has been asked whether similar classifications exist for algebras --- in particular, whether certain classes of purely infinite simple algebras can be classified by (algebraic) $K$-theory.  When examining this question, one immediately encounters two issues: First, there are many examples of purely infinite simple algebras for which  such a classification does not hold; so one must restrict attention to certain subclasses of purely infinite simple algebras, possibly ones that are somehow very ``similar" or ``close" to being $C^*$-algebras.  Second, unlike topological $K$-theory, where Bott periodicity implies that all $K$-groups other than the $K_0$-group and the $K_1$-group are redundant, the algebraic $K_n$-groups can be distinct for every $n \in \N$, and thus we may need to include all the algebraic $K$-groups in the invariant.  This second issue also raises the concern of whether such a classification is tractable --- if one does in fact need the algebraic $K_n$-groups for any $n \in \N$, then one must ask: How easily can the algebraic $K_n$-groups be calculated for a given algebra?

Recently a great deal of progress has been obtained in classifying Leavitt path algebras, which are the algebraic analogue of graph $C^*$-algebras, using algebraic $K$-theory.  If $E$ is a (directed) graph, the graph $C^*$-algebra $C^*(E)$ is constructed from $E$ in analogy with the construction of Cuntz-Krieger algebras.  When $C^*(E)$ is separable, purely infinite and simple, $C^*(E)$ falls into the class of $C^*$-algebras to which the Kirchberg-Phillips theorem applies, and $C^*(E)$ is classified up to Morita equivalence by the pair $(K^\textnormal{top}_0(C^*(E)), K^\textnormal{top}_1(C^*(E)))$.  Furthermore, if $E$ is also a finite graph, then  $K^\textnormal{top}_0(C^*(E))$ is a finitely generated abelian group and $K^\textnormal{top}_1(C^*(E))$ is isomorphic to the free part of $K^\textnormal{top}_0(C^*(E))$, so that $C^*(E)$ is classified up to Morita equivalence by the single group $K^\textnormal{top}_0(C^*(E))$.

On the algebraic side, as shown in \cite{ap:JA}, \cite{amp}, and \cite{ap} one may perform a similar construction to produce algebras from graphs.  If $E$ is a graph and $\mathsf{k}$ is any field, then one may mimic the graph $C^*$-algebra construction to produce a $\mathsf{k}$-algebra $L_\mathsf{k} (E)$, which is called the \emph{Leavitt path algebra} of $E$ over $\mathsf{k}$.  For a given graph $E$, the algebra $L_\mathsf{k} (E)$ (for any field $\mathsf{k}$) has many properties in common with $C^*(E)$.  Using results from symbolic dynamics, Abrams, Louly, Pardo, and Smith showed in \cite[Theorem~1.25]{alps} that if $E$ is a finite graph and $L_\mathsf{k} (E)$ is purely infinite and simple, then $L_\mathsf{k} (E)$ is determined up to Morita equivalence by the pair consisting of the group $K_0^\textnormal{alg} (L_\mathsf{k} (E))$ and the value $\sign (\det (I-A_E^t))$, where $A_E$ is the vertex matrix of the graph $E$.  The authors of \cite{alps} were unable to determine if the ``sign of the determinant" is a necessary part of the invariant; i.e., whether the algebraic $K_0$-group alone suffices to classify $L_\mathsf{k}(E)$.  This is currently a question of intense interest in the subject of Leavitt path algebras; in particular:  Do there exist finite graphs with purely infinite simple Leavitt path algebras that are Morita equivalent but for which the signs of the determinants are different?  It is also interesting to note that the field $\mathsf{k}$ does not appear in the invariant, and --- similar to the graph $C^*$-algebra situation --- the only algebraic $K$-group needed is the $K_0$-group.  Thus the invariant $(K_0^\textnormal{alg} (L_\mathsf{k} (E)), \sign (\det (I-A_E^t)) )$ can be easily computed from $E$ and does not depend on $\mathsf{k}$.

Building on recent work of S\o rensen in \cite{sorensen}, the second and third authors showed in \cite[Theorem~7.4]{rt} that when $E$ is a countable graph with a finite number of vertices but an infinite number of edges, then the complete Morita equivalence invariant for $L_\mathsf{k}(E)$ is the pair consisting of the group $K^\textnormal{alg}_0 (L_\mathsf{k} (E))$ and the number of singular vertices of $E$, often denoted $|E^0_\textnormal{sing}|$.  (Recall that a vertex is called \emph{singular} if it either emits no edges or infinitely many edges.)  There are a few interesting things to note here: First, the ``sign of the determinant" obstruction disappears in this situation and, unlike the finite graph case, we know the pair $(K^\textnormal{alg}_0 (L_\mathsf{k} (E)), |E^0_\textnormal{sing}|)$ is a complete Morita equivalence invariant with no pieces that are possibly extraneous.  Second, the group $K^\textnormal{alg}_0 (L_\mathsf{k} (E))$ can be easily calculated and the number of singular vertices can be easily determined from $E$, so the entire invariant can be computed in a tractable way.  Also, as in the finite graph situation, the field $\mathsf{k}$ does not play a role in the invariant.

While the result of \cite[Theorem~7.4]{rt} provides an easily computable complete Morita equivalence invariant for certain purely infinite simple Leavitt path algebras, there is one aspect of the classification that is not completely satisfying: The number of singular vertices is a property of the graph $E$, and not an algebraic property of $L_\mathsf{k}(E)$.  This means that the invariant is described (and computed) from the way $L_\mathsf{k}(E)$ is being presented, not from an intrinsic algebraic property of $L_\mathsf{k}(E)$.  Consequently, if one wants to generalize the classification to include purely infinite simple algebras that are not constructed from graphs, it is unclear what the invariant should be and what one should use in place of the number of singular vertices.  Thus one may ask the following two questions: Can the invariant be reformulated in terms of the $K$-groups?  If so, will the algebraic $K_0$-group and $K_1$-group suffice (as with graph $C^*$-algebras), or will one need to include the higher algebraic $K$-groups?

These questions were partially answered in \cite{rt}, and interestingly it is found that it depends on the underlying field.  It was shown that if $\mathsf{k}$ is a field with ``no free quotients" (see \cite[Definition~6.1 and Definition~6.9]{rt}), then the pair of algebraic $K$-groups $K^\textnormal{alg}_0 (L_\mathsf{k} (E))$ and $K^\textnormal{alg}_1 (L_\mathsf{k} (E))$ provides a complete Morita equivalence invariant for $L_\mathsf{k}(E)$ \cite[Theorem~8.6]{rt}.  The class of fields with no free quotients includes such fields as $\R$, $\C$, all finite fields, and all algebraically closed fields; so this result applies in a number of situations.  In this case only the algebraic $K_0$-group and $K_1$-group are needed (in analogy with the graph $C^*$-algebra case), and moreover, both $K^\textnormal{alg}_0 (L_\mathsf{k} (E))$ and $K^\textnormal{alg}_1 (L_\mathsf{k} (E))$ can be easily calculated for any graph $E$ and field $\mathsf{k}$.  Interestingly, the field $\mathsf{k}$ comes up in the computation of $K^\textnormal{alg}_1 (L_\mathsf{k} (E))$ so that, unlike our prior situations, $\mathsf{k}$ does play a role in the invariant.

Although many fields are fields with no free quotients, the field of rational numbers $\Q$ is an example of a field that is not.  Moreover, \cite[Example~11.2]{rt} gives examples of countable graphs $E$ and $F$, each having a finite number of vertices and infinite number of edges and with purely infinite simple Leavitt path algebras, such that $K^\textnormal{alg}_0 (L_\Q(E)) \cong K^\textnormal{alg}_0 (L_\Q (F))$ and $K^\textnormal{alg}_1 (L_\Q (E)) \cong K^\textnormal{alg}_1 (L_\Q (F))$, but $L_\Q (E)$ is not Morita equivalent to $L_\Q (F)$.  This shows that for  the algebraic $K_0$-group and $K_1$-group to provide a complete Morita equivalence invariant requires a hypothesis on the underling field, and this pair of $K$-groups does not determine the Morita equivalence class of $L_\mathsf{k} (E)$ in general.

Given the findings of \cite[Example~11.2]{rt}, it is natural to ask if the problems encountered could be avoided by including additional algebraic $K$-groups in the invariant.  Specifically, one may ask the following:

$ $

\noindent \textbf{Question~1:} Can we determine classes of fields with the property that some collection of algebraic $K$-groups provides a complete Morita equivalence invariant for purely infinite simple Leavitt path algebras of the form $L_\mathsf{k} (E)$, where $\mathsf{k}$ is a field from this class and $E$ is a graph with a finite number of vertices and an infinite number of edges?

$ $

\noindent The result of \cite[Theorem~8.6]{rt} shows that for fields with no free quotients the $K_0$-group and $K_1$-group suffice.  But what about other classes?  In particular, can some collection of algebraic $K$-groups be used to classify such Leavitt path algebras when the field is $\Q$?

$ $

If higher algebraic $K$-groups are needed in the invariant, then this also naturally raises the following question:

$ $

\noindent \textbf{Question~2:}  Can we compute $K_n^\textnormal{alg}(L_\mathsf{k} (E))$ for $n \in \Z$?

$ $

\noindent Explicit formulas exist for $K_0^\textnormal{alg}(L_\mathsf{k} (E))$ and $K_1^\textnormal{alg}(L_\mathsf{k} (E))$, but there are currently no explicit formulas for other algebraic $K$-groups of Leavitt path algebras.  Even if such formulas could be given in particular situations (e.g., under hypotheses on the field, or the graph, or the algebra) it would improve the current state of affairs.

The work in this paper is motivated by Question~1 and Question~2 above, as well as a desire to better understand the phenomenon encountered in \cite[Example~11.2]{rt} and find a complete Morita equivalence invariant for when the underlying field is $\Q$.

In this paper we do the following: In Section~2 we review definitions and establish notation. In Section~3 we extend the long exact sequence of \cite{abc}, which relates the $K$-groups of a Leavitt path algebra to the $K$-groups of its underlying field, to include all countable graphs (and, in particular, graphs that have infinite emitters).  We then use this long exact sequence to determine explicit formulas for $K^\textnormal{alg}_n(L_\mathsf{k}(E))$ for $n \leq 1$ when $E$ is any graph and $\mathsf{k}$ is any field.  In Section~4 we give explicit formulas for all algebraic $K$-groups of Leavitt path algebras over finite fields (see Theorem~\ref{LPA-K-theory-finite-field}).  In Section~5 we give an explicit formula for $K^\textnormal{alg}_n(L_\mathsf{k}(E))$ under the hypothesis that either $K^\textnormal{alg}_n(k)$ is divisible or $K^\textnormal{alg}_{n-1}(k)$ is free abelian (see Theorem~\ref{algclosed}).  In particular, algebraically closed fields satisfy this hypothesis, and we are able to thereby give explicit formulas for all algebraic $K$-groups of a Leavitt path algebra over an algebraically closed field.  In Section~6 we examine the rank and corank for abelian groups, and observe that a field $\mathsf{k}$ has no free quotients if and only if $\corank K_1^\textnormal{alg} (\mathsf{k}) = 0$.  In Section~7 we define size functions as generalizations of corank and exact size functions as generalizations of rank.  In Section~8 we consider Question~1 above, and we show that if $\mathsf{k}$ is a field and $E$ is a countable graph with a finite number of vertices and an infinite number of edges for which $L_\mathsf{k} (E)$ is purely infinite and simple, then provided one of the following conditions holds:
\begin{itemize}
\item[(i)] there is a size function $F$ and a natural number $n \in \N$ such that $F(K_n^\textnormal{alg}(\mathsf{k})) = 0$, and $0 < F(K^\textnormal{alg}_{n-1}(k)) < \infty$; or
\item[(ii)] there is an exact size function $F$ and a natural number $n \in \N$ such that $F(K_n^\textnormal{alg}(\mathsf{k})) < \infty$, and $0 < F(K^\textnormal{alg}_{n-1}(k)) < \infty$
\end{itemize}
the pair consisting of $K_0^\textnormal{alg} ( L_\mathsf{k} (E))$ and $K_n^\textnormal{alg} ( L_\mathsf{k} (E))$ is a complete Morita equivalence invariant for $L_\mathsf{k} (E)$ (see Corollary~\ref{K0-Kn-1-Kn-exact-Mor-Eq-cor} and Corollary~\ref{K0-Kn-1-Kn-size-fct-Mor-Eq-cor}).  The result of \cite{rt}, which says that the algebraic $K_0$-group and $K_1$-group provide a complete invariant over fields with no free quotients, is a special case of this result.  In addition, this result implies (see Theorem~\ref{number-field-sing-thm}) that if $\mathsf{k}$ is a number field (i.e., a finite field extension of $\Q$), then the pair consisting of $K_0^\textnormal{alg} ( L_\mathsf{k} (E))$ and $K_6^\textnormal{alg} ( L_\mathsf{k} (E))$ is a complete Morita equivalence invariant.  In particular, since $\Q$ itself is a number field, this implies that the pair consisting of $K_0^\textnormal{alg} ( L_\Q (E))$ and $K_6^\textnormal{alg} ( L_\Q(E))$ is a complete Morita equivalence invariant for $L_\Q(E)$ when $L_\Q(E)$ is purely infinite simple  and $E$ has a finite number of vertices and an infinite number of edges.  

The authors thank Rasmus Bentmann for useful and productive conversations while this work was being performed.

\section{Preliminaries}

We write $\N := \{ 1, 2, \ldots \}$ for the natural numbers and $\Z^+ := \{ 0, 1, 2, \ldots \}$ for the non-negative integers.  We use the symbol $\mathsf{k}$ to denote a field.  If $R$ is a ring and $n \in \Z$, we let $K_n(R)$ denote the $n$\textsuperscript{th} algebraic $K$-group of $R$.  (Unlike in the introduction of this paper, we will drop the ``\textnormal{alg}" superscript on the $K$-groups, since we will only be considering algebraic $K$-theory.)

A \emph{graph} $(E^0, E^1, r, s)$ consists of a set $E^0$ of vertices, a  set $E^1$ of edges, and maps $r: E^1 \to E^0$ and $s: E^1 \to E^0$ that identify the range and source of each edge. What we call a graph is often referred to as a \emph{directed graph} or a \emph{quiver} in other literature.  A graph is \emph{countable} if both the sets $E^0$ and $E^1$ are countable, and a graph is \emph{finite} if both the sets $E^0$ and $E^1$ are finite. 

$ $

\noindent \textbf{Standing Assumption:} Throughout this paper all graphs are assumed to be countable.

$ $

Let $E=(E^0, E^1, r, s)$ be a graph and let $v \in E^0$ be a vertex of $E$. We say $v$ is a \emph{sink} if $s^{-1} (v) = \emptyset$, and we say $v$ is an \emph{infinite emitter} if $|s^{-1} (v)| = \infty$. A \emph{singular vertex} is a vertex that is either a sink or an infinite emitter, and we denote the set of singular vertices by $E^0 _{\sing}$. We let $E^0 _{\reg} : = E^0 \setminus E^0 _{\sing}$ and refer to an element of $E^0 _{\reg}$ as a \emph{regular vertex}. Note that a vertex $v \in E^0$ is a \emph{regular vertex} if and only if $0 < |s^{-1}(v) | < \infty$.

If $E$ is a graph, a \emph{path} is a sequence of edges $\alpha := e_1 e_2 \ldots e_n$ with $r(e_i) = s(e_{i+1})$ when $n \geq 2$ for $1 \leq i \leq n-1$. The \emph{length} of $\alpha$ is $n$, and we consider a single edge to be a path of length $1$, and a vertex to be a path of length $0$. The set of all paths in $E$ is denoted by $E^*$. A \emph{cycle} is a path $\alpha = e_1 e_2 \ldots e_n$ with $n \geq 1$ and $r( e_n ) = s( e_1)$.

If $\alpha = e_1e_2 \ldots e_n$ is a cycle, an \emph{exit} for $\alpha$ is an edge $f \in E^1$ such that $s(f) = s(e_i)$ and $f \neq e_i$ for some $i$.  A graph is said to satisfy \emph{Condition~(L)} if every cycle in the graph has an exit.  An \emph{infinite path} in a graph $E$ is an infinite  sequence of edges $\mu := e_1 e_2 \ldots$ with $r(e_i) = s(e_{i+1})$ for all $i \in \N$.  A graph $E$ is called \emph{cofinal} if whenever $\mu := e_1 e_2 \ldots$ is an infinite path in $E$ and $v \in E^0$, then there exists a finite path $\alpha \in E^*$ with $s(\alpha) = v$ and $r(\alpha) = s(e_i)$ for some $i \in \N$.

We say that a graph $E$ is \emph{simple} if $E$ satisfies all of the following three conditions:
\begin{itemize}
\item[(i)] $E$ is cofinal, 
\item[(ii)] $E$ satisfies Condition~(L), and 
\item[(iii)] whenever $v \in E^0$ and $w \in E^0_\textnormal{sing}$, there exists a path $\alpha \in E^*$ with $s(\alpha) = v$ and $r(\alpha) = w$.  
\end{itemize}
It is proven in \cite[Proposition~4.2]{rt} that the following are equivalent:
\begin{itemize}
\item[(1)] The graph $E$ is simple.
\item[(2)] The Leavitt path algebra $L_\mathsf{k} (E)$ is simple for any field $\mathsf{k}$.
\item[(3)] The graph $C^*$-algebra $C^*(E)$ is simple.
\end{itemize}

Given a graph $E$, the \emph{vertex matrix} $A_E$ is the $E^0 \times E^0$ matrix whose entries are given by 
$A_E(v,w) := | \{ e \in E^1 : s(e) = v \text{ and } r(e) = w \}|$.  We write $\infty$ for this value when $\{ e \in E^1 : s(e) = v \text{ and } r(e) = w \}$ is an infinite set, so $A_E$ takes values in $\Z^+ \cup \{ \infty \}$.

Let $E$ be a graph, and let $\mathsf{k}$ be a field. We let $(E^1)^*$
denote the set of formal symbols $\{ e^* : e \in E^1 \}$.  The \emph{Leavitt path
algebra of $E$ with coefficients in $\mathsf{k}$}, denoted $L_\mathsf{k}(E)$,  is
the free associative $\mathsf{k}$-algebra generated by a set $\{v : v \in
E^0 \}$ of pairwise orthogonal idempotents, together with a set
$\{e, e^* : e \in E^1\}$ of elements, modulo the ideal generated
by the following relations:
\begin{enumerate}
\item $s(e)e = er(e) =e$ for all $e \in E^1$
\item $r(e)e^* = e^* s(e) = e^*$ for all $e \in E^1$
\item $e^*f = \delta_{e,f} \, r(e)$ for all $e, f \in E^1$
\item $v = \displaystyle \sum_{\{e \in E^1 : s(e) = v \}} ee^*$ whenever $v \in E^0_\reg$.
\end{enumerate}

If $\alpha = e_1 \ldots e_n$ is a path of positive length, we define $\alpha^* = e_n^* \ldots e_1^*$.  One can show that $$L_\mathsf{k}(E) = \operatorname{span}_\mathsf{k} \{ \alpha \beta^* : \text{$\alpha$ and $\beta$ are paths in $E$ with $r(\alpha) = r(\beta)$} \}.$$ 
If $E$ is a graph and $\mathsf{k}$ is a field, the Leavitt path algebra $L_\mathsf{k}(E)$ is unital if and only if the vertex set $E^0$ is finite, in which case $1 = \sum_{v \in E^0} v$.

Suppose $E$ is a graph with a singular vertex $v_0 \in E^0 _\sing$. We \emph{add a tail} at $v_0$ by attaching a graph of the form
$$\xymatrix{
v_0 \ar[r] & v_1 \ar[r] & v_2 \ar[r] & v_3 \ar[r] & \cdots
}$$
to $E$ at $v_0$, and in the case $v_0$ is an infinite emitter, we list the edges of $s^{-1} (v_0)$ as $g_1, g_2, g_3, \ldots$, remove the edges in $s^{-1} (v_0)$, and for each $g_j$ we draw an edge $f_j$ from $v_{j-1}$ to $r (g_j)$. A \emph{desingularization} of $E$ is a graph $F$ obtained by adding a tail at every singular vertex of $E$.  It is shown in \cite[Theorem~5.2]{ap} that if $F$ is a desingularization of $E$, then for any field $\mathsf{k}$ the Leavitt path algebra $L_\mathsf{k} (E)$ is isomorphic to a full corner of the Leavitt path algebra $L_\mathsf{k} (F)$, and the algebras $L_\mathsf{k} (E)$ and $L_\mathsf{k} (F)$ are Morita equivalent.

\section{The Long Exact Sequence for $K$-groups of Leavitt Path Algebras}

In this section we extend the $K$-theory computation of \cite[Theorem~7.6]{abc} to Leavitt path algebras of graphs that may contain singular vertices.  For an abelian group $G$ and a set $S$ (possibly infinite), we denote the direct sum $\bigoplus_{ S } G$ by $G^{S}$.  Note that this differ from the notation used in \cite{abc} in which the authors denoted the direct sum by $G^{(S)}$.

\begin{theorem}\label{abc}
Let $E = (E^0, E^1, r, s)$ be a graph. Decompose the vertices of $E$ as $E^0 = E_{\reg} ^0 \sqcup E_{\sing} ^0$, and with respect to this decomposition write the vertex matrix of $E$ as
$$\left( \begin{matrix} B_E & C_E \\ * & * \end{matrix} \right)$$
where $B_E$ and $C_E$ have entries in $\Z ^+$ and each $*$ has entries in $\Z ^+ \cup \{ \infty \}$.
If $\mathsf{k}$ is a field, then for the Leavitt path algebra $L_\mathsf{k} (E)$ there is a long exact sequence 
$$ \xymatrix{ \cdots \ar[r] &
K_n (\mathsf{k})^{E^0 _{\reg}} \ar[r]^{\left( \begin{smallmatrix} B_E ^t - I \\ C_E ^t \end{smallmatrix} \right)} & K_n (\mathsf{k})^{E^0} \ar[r] & K_n (L_\mathsf{k}(E)) \ar[r] & K_{n-1} (\mathsf{k}) ^{E^0 _{\reg}} \ar[r]^-{\left( \begin{smallmatrix} B_E ^t - I \\ C_E ^t \end{smallmatrix} \right)} & \cdots\\} $$
for all $n \in \Z$.
\end{theorem}

\begin{proof}
If $E$ has no singular vertices, then the result holds by \cite[Theorem~7.6]{abc}. Otherwise, we will apply \cite[Theorem~7.6]{abc} to a desingularization of $E$ and show that the result holds for $E$ as well. Suppose $E$ has at least one singular vertex. List the singular vertices of $E$ as $E_{\sing} ^0 := \{v_1 ^0, v_2 ^0, v_3 ^0, \ldots \}.$  Note that $E_{\sing} ^0$ could be finite or countably infinite, but not empty.

Let $F$ be a desingularization of $E$. In forming $F$ from $E$, we add a tail to each singular vertex of $E$ and ``distribute" the edges of each infinite emitter along the vertices of the tail added to that infinite emitter.   For each singular vertex $v_i^0 \in E_{\sing} ^0$, let $\{v_i ^1, v_i ^2, v_i ^3 , \ldots \}$ denote the vertices of the tail added to $v_i^0$. (See  \cite[Section~2]{dt} for details.) If $A_F$ is the vertex matrix of $F$, we will now describe $A_F ^t - I$ following \cite[Lemma~2.3]{dt2}. For each $1 \leq i \leq |E_{\sing} ^0|$, let $D_i$ denote the $E_{\sing} ^0 \times \N$ matrix with $1$ in the $(i,1)$ position and zeros elsewhere:
$$ D_i = \left( \begin{matrix} 0 & 0 & 0 & \cdots \\ 
\vdots & \vdots & \vdots & \cdots \\
0 & 0 & 0 & \cdots \\
1 & 0 & 0 & \cdots \\
 0 & 0 & 0 & \cdots \\ 
\vdots & \vdots & \vdots & \ddots \end{matrix} \right).$$
Let $Z$ denote the $\N \times \N$ matrix with $-1$ in each entry of the diagonal and $1$ in each entry of the superdiagonal:
$$Z = \left( \begin{matrix} -1 & 1 & 0 & 0 & \cdots \\
0 & -1 & 1 & 0 & \cdots \\
0 & 0 & -1 & 1 & \cdots \\
\vdots & \vdots & \vdots & \vdots & \ddots \end{matrix} \right).$$
With respect to the decomposition $E_{\reg} ^0 \sqcup E_{\sing} ^0 \sqcup  \{v_1 ^1, v_1 ^2, v_1 ^3, \ldots \} \sqcup \{v_2 ^1, v_2 ^2, v_2 ^3, \ldots \}\sqcup \cdots$, we have
$$A_F ^t - I = \left( \begin{matrix} B_E ^t - I & X_1 ^t & X_2 ^t & X_3 ^t & \cdots \\
C_E^t  & Y_1 ^t - I & Y_2 ^t & Y_3 ^t & \cdots \\
0 & D_1 ^t & Z^t & 0 & \cdots \\
0 & D_2 ^t & 0 & Z^t & \cdots \\
\vdots & \vdots & \vdots & \vdots & \ddots \end{matrix} \right)$$
where each $X_i ^t$ and each $Y_i ^t$ is column-finite.

Since $E^0 \subseteq F^0$, we may define $\iota_n : K_n (\mathsf{k})^{E^0} \to K_n (\mathsf{k}) ^{F^0}$ and $\iota^{\reg} _n: K_n (\mathsf{k})^{E^0 _{\reg}} \to K_n (\mathsf{k}) ^{F^0}$ to be the inclusion maps.  If  $\textbf{x} \in K_n (\mathsf{k}) ^{E_{\reg} ^0}$, then
$$(A_F ^t - I)~ \iota^{\reg} _n (\textbf{x})=(A_F ^t- I)\begin{pmatrix} \mathbf{x} \\ \mathbf{0} \\ \begin{pmatrix} \textbf{0} \\ \textbf{0} \\ \vdots \end{pmatrix} \end{pmatrix} = \begin{pmatrix} (B_E ^t - I)\textbf{x} \\ C_E ^t \textbf{x} \\ \begin{pmatrix} \textbf{0} \\ \textbf{0} \\ \vdots \end{pmatrix} \end{pmatrix}= 
\iota_n \left( \left( \begin{matrix} B_E ^t - I \\ C_E ^t \end{matrix} \right )\textbf{x} \right).$$
So, the diagram 
\begin{equation} \label{inclusion-commute-eq}
\xymatrix{
K_n(\mathsf{k})^{E^0_{\reg}} \ar[d]_{\iota^{\reg} _n} \ar[r]^-{\left( \begin{smallmatrix} B_E ^t - I \\ C_E ^t \end{smallmatrix} \right)}
& K_n(\mathsf{k})^{E^0} \ar[d]_{\iota_n}\\ K_n(\mathsf{k})^{F^0} \ar[r]^-{A_F ^t - I}& K_n(\mathsf{k})^{F^0} }
\end{equation}
commutes for each $n \in \Z$.

By \cite[Theorem~5.2]{ap}, there is an embedding $\phi : L_\mathsf{k} (E) \to L_\mathsf{k}(F)$  onto a full corner of  $L_\mathsf{k} (F)$.  By \cite[Lemma~5.2]{rtideal} $\phi$ induces an isomorphism $\phi_* : K_n (L_\mathsf{k} (E)) \to K_n (L_\mathsf{k} (F))$ for each $n \in \Z$. Since $F$ has no singular vertices, \cite[Theorem~7.6]{abc} implies there exists a long exact sequence
\begin{equation} \label{desing-long-eq}
\xymatrix{
 \cdots \ar[r] 
& K_n (\mathsf{k})^{F ^0} \ar[r]^-{A_F ^t - I} 
& K_n (\mathsf{k})^{F^0} \ar[r]^-f
 & K_n (L_\mathsf{k} (F)) \ar[r]^-g
 & K_{n-1} (\mathsf{k}) ^{F ^0} \ar[r]^-{A_F ^t - I}
 & \cdots}
 \end{equation}
Combining \eqref{desing-long-eq} with \eqref{inclusion-commute-eq} and the isomorphisms $\phi_* : K_n (L_\mathsf{k} (E)) \to K_n (L_\mathsf{k} (F))$ we obtain a commutative diagram
$$\xymatrix{ \cdots  & K_n (\mathsf{k})^{E^0 _{\reg}} \ar[r]^-{\left( \begin{smallmatrix} B_E ^t - I \\ C_E ^t \end{smallmatrix} \right)} \ar[d]^{\iota^{\reg} _n} 
& K_n (\mathsf{k})^{E^0} \ar[d]^{\iota _n}
& K_n (L_\mathsf{k} (E)) \ar[d]^{\phi_*} 
& K_{n-1}(\mathsf{k})^{E^0_{\reg}} \ar[r]^-{\left( \begin{smallmatrix} B_E ^t - I \\ C_E ^t \end{smallmatrix} \right)} \ar[d]^{\iota^{\reg} _{n-1}} & \cdots\\
 \cdots \ar[r] 
& K_n (\mathsf{k})^{F ^0} \ar[r]^-{A_F ^t - I} 
& K_n (\mathsf{k})^{F^0} \ar[r]^-f
 & K_n (L_\mathsf{k} (F)) \ar[r]^-g
 & K_{n-1} (\mathsf{k}) ^{F ^0} \ar[r]^-{A_F ^t - I}
 & \cdots}$$
with the lower row exact.

Define $f_0 : K_n(\mathsf{k})^{E^0} \to K_n(L_\mathsf{k} (E))$ by $f_0 : = \phi_* ^{-1} \circ f \circ \iota_n$. Define $g_0 : K_n (L_\mathsf{k} (E)) \to K_{n-1} (\mathsf{k})^{E_{\reg} ^0}$ by $g_0 := \pi^{\reg} _{n-1} \circ g \circ \phi_*$, where $\pi^{\reg}_n : K_n (\mathsf{k})^{F ^0} \to K_n(\mathsf{k})^{E^0 _{\reg}}$ by $$\pi^{\reg}_n \begin{pmatrix} \textbf{x} \\ \textbf{y} \\ \begin{pmatrix} \textbf{z}_1 \\ \textbf{z}_2 \\ \vdots \end{pmatrix} \end{pmatrix} = \textbf{x}.$$
Altogether, we have the diagram
\begin{equation} \label{main} \xymatrix{ \cdots \ar[r] & K_n (\mathsf{k})^{E^0 _{\reg}} \ar[r]^-{\left( \begin{smallmatrix} B_E ^t - I \\ C_E ^t \end{smallmatrix} \right)} \ar[d]^{\iota^{\reg} _n} 
& K_n (\mathsf{k})^{E^0} \ar[d]^{\iota _n} \ar[r]^-{f_0} 
& K_n (L_\mathsf{k} (E)) \ar[d]^{\phi_*} \ar[r]^-{g_0}
& K_{n-1}(\mathsf{k})^{E^0_{\reg}} \ar[r]^-{\left( \begin{smallmatrix} B_E ^t - I \\ C_E ^t \end{smallmatrix} \right)} \ar[d]^{\iota^{\reg} _{n-1}} & \cdots\\
 \cdots \ar[r] 
& K_n (\mathsf{k})^{F ^0} \ar[r]^-{A_F ^t - I} 
& K_n (\mathsf{k})^{F^0} \ar[r]^-f
 & K_n (L_\mathsf{k} (F)) \ar[r]^-g
 & K_{n-1} (\mathsf{k}) ^{F ^0} \ar[r]^-{A_F ^t - I}
 & \cdots} \end{equation}
and will show that this diagram commutes and that the upper row is exact. To do so, we will frequently need:
\begin{equation} \label{incl}\im g \subseteq \im \iota^{\reg} _{n-1} ,\end{equation}
so we prove this now. Let $(\textbf{u}, \textbf{v}, (\textbf{w}_1, \textbf{w}_2, \ldots))^t \in \ker (A_F ^t - I)=\im g \subseteq K_{n-1} (\mathsf{k}) ^{F^0}$ written with respect to the decomposition $F^0 = E_{\reg} ^0 \sqcup E_{\sing} ^0 \sqcup  \{v_1 ^1, v_1 ^2, v_1 ^3, \ldots \} \sqcup \{v_2 ^1, v_2 ^2, v_2 ^3, \ldots \}\sqcup \cdots$. Then
$$
\begin{pmatrix} \mathbf{0} \\ \mathbf{0} \\ \begin{pmatrix} \mathbf{0} \\ \mathbf{0} \\ \vdots \end{pmatrix} \end{pmatrix}
=
(A_F ^t - I) \left(\begin{matrix} \mathbf{u} \\ \mathbf{v} \\ \begin{pmatrix} \mathbf{w}_1 \\ \mathbf{w}_2 \\ \vdots \end{pmatrix} \end{matrix}\right) 
=
\left(\begin{matrix}
(B_E ^t - I)\mathbf{u} + X_1 ^t \mathbf{v} + X_2 ^t \mathbf{w}_1 + X_3 ^t \mathbf{w}_2+ \cdots \\
C_E ^t \mathbf{u} + (Y_1 ^t - I)\mathbf{v} + Y_2 ^t \mathbf{w}_1+ Y_3 ^t \mathbf{w}_2+ \cdots \\
D_1 ^t \mathbf{v} + Z^t \mathbf{w}_1 \\
D_2 ^t \mathbf{v}+ Z^t \mathbf{w}_2 \\
D_3 ^t \mathbf{v}+ Z^t \mathbf{w}_3 \\
\vdots
\end{matrix}\right).
$$
We have $D_i ^t \mathbf{v} + Z\mathbf{w}_i = 0$ for all $1 \leq i \leq |E_{\sing} ^0|$. For a fixed $i \in \N$, define $\mathbf{v}: = (v_1, v_2, \ldots)$ and $\mathbf{w}_i : = (w_{i,1}, w_{i,2}, \ldots)$. Then $D_i ^t \mathbf{v} + Z^{t}\mathbf{w}_i = 0$ implies
$$\begin{pmatrix} v_i \\ 0 \\ 0 \\ \vdots \end{pmatrix} + \begin{pmatrix} -w_{i,1} \\ w_{i,1} - w_{i,2} \\ w_{i,2} - w_{i,3} \\ \vdots \end{pmatrix}=\begin{pmatrix} 0 \\ 0 \\ 0 \\ \vdots \end{pmatrix},$$ and we conclude that $v_i = w_{i,1} = w_{i,2} = w_{i,3} = \cdots$ for each $1 \leq i \leq |E_{\sing} ^0|$. Since $\mathbf{w}_i$ is in the direct sum $K_{n-1} (\mathsf{k})^{\{ v_i ^1, v_i ^2 , \ldots \}},$ the entries $w_{i,k}$ are eventually $0$ for each $i$, so $\mathbf{0} = \mathbf{v} = \mathbf{w}_1 = \mathbf{w}_2 = \cdots$. This implies $\im g = \ker(A_F ^t -I) \subseteq \im \iota_n ^{\reg}$.

Now we check commutativity of \eqref{main}. We already have that $(A_F ^t - I) \circ \iota^{\reg} _n = \iota_n \circ \left( \begin{smallmatrix} B_E ^t - I \\ C_E ^t \end{smallmatrix} \right )$ from \eqref{inclusion-commute-eq}. From the definition of $f_0$ we have $\phi_* \circ f_0 = \phi_* \circ \phi_* ^{-1} \circ f \circ \iota_n = f \circ \iota_n$. Finally, $\iota_{n-1} ^{\reg} \circ g_0 = \iota_{n-1} ^{\reg} \circ \pi _{n-1} ^{\reg} \circ g \circ \phi_* = g \circ \phi_*$ from the definition of $g_0$ and by \eqref{incl}. Hence \eqref{main} is commutative.

Next, we verify exactness at $K_n (\mathsf{k})^{E^0}$, $K_n (L_{\mathsf{k}} (E))$, and $K_{n-1} (\mathsf{k})^{E^0_{\mathrm{reg}}}.$

\noindent \underline{Step 1}: $\im\left( \begin{smallmatrix} B_E ^t - I \\ C_E ^t \end{smallmatrix} \right) = \ker f_0$. Because of the commutativity and exactness of \eqref{desing-long-eq}:
$$ f_0 \circ \left( \begin{smallmatrix} B_E ^t - I \\ C_E ^t \end{smallmatrix} \right) (\mathbf{x}) 
 =\phi_* ^{-1} \circ f \circ \iota_n \circ \left( \begin{smallmatrix} B_E ^t - I \\ C_E ^t \end{smallmatrix} \right) (\mathbf{x})
=\phi_* ^{-1} \circ f \circ (A_F ^t - I) \circ \iota^{\reg} _n (\mathbf{x}) 
= 0
$$
for all $\mathbf{x} \in K_n(\mathsf{k})^{E^0_{\reg}}$. Hence $\im\left( \begin{smallmatrix} B_E ^t - I \\ C_E ^t \end{smallmatrix} \right) \subseteq \ker f_0$. For the reverse inclusion, if $(\mathbf{x}, \mathbf{y})^t \in \ker f_0$, then $ \phi_* ^{-1} \circ f \circ \iota_n ( (\mathbf{x}, \mathbf{y})^t) = 0$, and since $\phi_*$ is an isomorphism, $\iota_n ((\mathbf{x}, \mathbf{y})^t) \in \ker f = \im(A_F ^t - I)$. Let $(\mathbf{u}, \mathbf{v}, (\mathbf{w}_1, \mathbf{w}_2, \ldots))^t \in K_n (\mathsf{k}) ^{F^0}$ be an element that $A_F ^t - I$ maps to $\iota_n ((\mathbf{x}, \mathbf{y})^t)$, written with respect to the decomposition $F^0 = E_{\reg} ^0 \sqcup E_{\sing} ^0 \sqcup  \{v_1 ^1, v_1 ^2, v_1 ^3, \ldots \} \sqcup \{v_2 ^1, v_2 ^2, v_2 ^3, \ldots \}\sqcup \cdots$. Then
$$\left(\begin{matrix} \mathbf{x} \\ \mathbf{y} \\ \begin{pmatrix} \mathbf{0} \\ \mathbf{0} \\ \vdots \end{pmatrix} \end{matrix}\right)= \iota_n \begin{pmatrix} \mathbf{x} \\ \mathbf{y} \end{pmatrix} =
(A_F ^t - I) \left(\begin{matrix} \mathbf{u} \\ \mathbf{v} \\ \begin{pmatrix} \mathbf{w}_1 \\ \mathbf{w}_2 \\ \vdots \end{pmatrix} \end{matrix}\right) =
\left(\begin{matrix}
(B_E ^t - I)\mathbf{u} + X_1 ^t \mathbf{v} + X_2 ^t \mathbf{w}_1 + \cdots \\
C_E ^t \mathbf{u} + (Y_1 ^t - I)\mathbf{v} + Y_2 ^t \mathbf{w}_1+ \cdots \\
D_1 ^t \mathbf{v} + Z^t \mathbf{w}_1 \\
D_2 ^t \mathbf{v}+ Z^t \mathbf{w}_2 \\
D_3 ^t \mathbf{v}+ Z^t \mathbf{w}_3 \\
\vdots
\end{matrix}\right).$$
Using the computations following \eqref{incl}, we obtain $\mathbf{0} = \mathbf{v} = \mathbf{w}_1 = \mathbf{w}_2 = \cdots$ and so
$$\left(\begin{matrix} \mathbf{x} \\ \mathbf{y} \\ \begin{pmatrix} \mathbf{0} \\ \mathbf{0} \\ \vdots \end{pmatrix} \end{matrix}\right)=
\left(\begin{matrix}
(B_E ^t - I)\mathbf{u} \\
C_E ^t \mathbf{u} \\
\left( \begin{matrix} \mathbf{0} \\
\mathbf{0} \\
\vdots \end{matrix} \right)
\end{matrix}\right).$$
Thus, $(\mathbf{x}, \mathbf{y})^t = \begin{pmatrix} B_E ^t - I \\ C_E ^t \end{pmatrix} \mathbf{u}$, so $(\mathbf{x}, \mathbf{y})^t \in \im \begin{pmatrix} B_E ^t - I \\ C_E ^t \end{pmatrix}$ and $\ker f_0 \subseteq \im\left(\begin{matrix} B_E ^t - I\\ C_E ^t \end{matrix}\right)$.

\smallskip

\noindent \underline{Step 2}: $\im f_0 = \ker g_0.$ Since $\im f= \ker g$,
$$g_0 \circ f_0 = \pi_{n-1} ^\reg \circ g \circ \phi_* \circ \phi_* ^{-1} \circ f \circ \iota_n = \pi_{n-1} ^\reg \circ g \circ f \circ \iota_n = 0,$$
implying that $\im f_0 \subseteq \ker g_0$. For the reverse inclusion, let $x \in \ker g_0$. Then $\mathbf{0} = g_0 (x) = \pi_{n-1} ^\reg \circ g \circ \phi_* (x),$ and by \eqref{incl} this implies $\phi_* (x) \in \ker g = \im f$. Let 
$$\left( \begin{matrix} \mathbf{p} \\ 
\mathbf{q} \\ 
\left(\begin{matrix}\mathbf{r}_1 \\ \mathbf{r}_2 \\ \vdots \end{matrix} \right) \end{matrix} \right)
\in K_n (\mathsf{k})^{F^0}$$
be an element that $f$ maps to $\phi_* (x),$ written with respect to the decomposition $F^0 = E_{\reg} ^0 \sqcup E_{\sing} ^0 \sqcup  \{v_1 ^1, v_1 ^2, v_1 ^3, \ldots \} \sqcup \{v_2 ^1, v_2 ^2, v_2 ^3, \ldots \}\sqcup \cdots$. For each $1 \leq i \leq |E_{\sing} ^0|$, write
$$\mathbf{r}_i = \left( \begin{matrix} r_{i,1} \\ r_{i,2} \\ r_{i,3} \\ \vdots \end{matrix} \right)$$
and define
$$b_i := \sum_{j=1} ^{\infty} r_{i,j} \quad \qquad \text{ and } \quad \qquad c_{i,k} := \sum_{j=k+1} ^{\infty} r_{i,j}.$$
Since $\mathbf{r}_i$ is in the direct sum $K_n (\mathsf{k})^{\{v_i ^1, v_i ^2, \ldots \}}$ for each $i$, each of the above sums has only finitely many nonzero terms. Also, since $(\mathbf{r}_1, \mathbf{r}_2, \ldots )^t$ is in the direct sum $K_n (\mathsf{k})^{F^0 \setminus E^0}$, we have that $\mathbf{r}_i = 0$ eventually, so
$$\mathbf{b} : = \left( \begin{matrix} b_1 \\ b_2 \\ \vdots \end{matrix} \right) \in K_n(\mathsf{k})^{E_{\sing} ^0} \quad \text { and } \quad \mathbf{c}:=\left( \begin{matrix} \mathbf{c}_1 \\ \mathbf{c}_2 \\ \vdots \end{matrix} \right) \in K_n(\mathsf{k}) ^{F^0 \setminus E^0} \text{ where } \mathbf{c}_i : = \left( \begin{matrix} c_{i,1} \\ c_{i,2} \\ \vdots \end{matrix} \right).$$
Now set
$$\mathbf{u}:= - X_1 ^t \mathbf{b} - X_2 ^t \mathbf{c}_1 - X_3 ^t \mathbf{c}_2 - \cdots$$
and
$$\mathbf{v}:= -(Y_1 ^t - I) \mathbf{b} - Y_2 ^t \mathbf{c}_1 - Y_3 ^t \mathbf{c}_2 - \cdots,$$
which are finite sums since $(\mathbf{c}_1, \mathbf{c}_2, \ldots)^t$ is in the direct sum.
This gives
$$(A_F ^t - I) \left( \begin{matrix} \mathbf{0} \\ \mathbf{b} \\ \left( \begin{matrix}  \mathbf{c}_1 \\ \mathbf{c}_2 \\ \vdots \end{matrix} \right) \end{matrix} \right) = \left( \begin{matrix} -\mathbf{u} \\ -\mathbf{v} \\ \left( \begin{matrix} \mathbf{r}_1 \\ \mathbf{r}_2 \\ \vdots \end{matrix} \right) \end{matrix} \right).$$
Thus, 
$$\left( \begin{matrix} -\mathbf{u} \\ -\mathbf{v} \\ \left( \begin{matrix} \mathbf{r}_1 \\ \mathbf{r}_2 \\ \vdots \end{matrix} \right) \end{matrix} \right) \in \im(A_F ^t - I) = \ker f,$$
and
$$f_0 \left( \begin{matrix} \mathbf{p} + \mathbf{u} \\ \mathbf{q} + \mathbf{v} \end{matrix} \right)= \phi_* ^{-1} \circ f \circ \iota_n \left( \begin{matrix} \mathbf{p} + \mathbf{u} \\ \mathbf{q} + \mathbf{v} \end{matrix} \right)
= \phi_* ^{-1} \circ f  \left( \begin{matrix} \mathbf{p} + \mathbf{u} \\ \mathbf{q} + \mathbf{v} \\ \begin{pmatrix} \mathbf{0} \\ \mathbf{0} \\ \vdots \end{pmatrix} \end{matrix}  \right)
= \phi_* ^{-1} \circ f \left( \begin{matrix} \mathbf{p} \\ \mathbf{q} \\ \left( \begin{matrix} \mathbf{r}_1 \\ \mathbf{r}_2 \\ \vdots \end{matrix} \right) \end{matrix} \right)
= x,$$
so $x \in \im f_0$.

\noindent \underline{Step 3}: $\im g_0 = \ker \left( \begin{smallmatrix} B_E ^t - I \\ C_E ^t \end{smallmatrix} \right).$ First we will show $\left( \begin{smallmatrix} B_E ^t - I \\ C_E ^t \end{smallmatrix} \right) \circ~g_0 = 0$, which happens if and only if $\iota_n \circ \left( \begin{smallmatrix} B_E ^t - I \\ C_E ^t \end{smallmatrix} \right) \circ g_0 = 0$ because $\iota_n$ is injective. Since
\begin{align*}\iota_n \circ \left( \begin{smallmatrix} B_E ^t - I \\ C_E ^t \end{smallmatrix} \right) \circ g_0 & =
(A_F ^t - I) \circ \iota_n ^\reg \circ g_0 \qquad \qquad \quad \,\,\,\, \text{(by commutativity of \eqref{inclusion-commute-eq})} \\
& = (A_F ^t - I) \circ \iota_n ^\reg \circ \pi_{n-1} ^\reg \circ g \circ \phi_*  \quad \text{(by definition of $g_0$)}\\
&=(A_F ^t -I) \circ g \circ \phi_*  \qquad \qquad \qquad \, \, \text{(by \eqref{incl})}\\
& = 0 \circ \phi_* = 0 \qquad \qquad \qquad \qquad \quad \,\, \text{(by exactness of \eqref{desing-long-eq})}, \end{align*}
it follows  $\im g_0 \subseteq \ker \left( \begin{smallmatrix} B_E ^t - I \\ C_E ^t \end{smallmatrix} \right).$ For the reverse inclusion, let $\mathbf{u} \in \ker\left( \begin{smallmatrix} B_E ^t - I \\ C_E ^t \end{smallmatrix} \right).$ By the commutativity of \eqref{inclusion-commute-eq} and exactness of \eqref{desing-long-eq}, we have $\iota^{\reg} _{n-1} (\mathbf{u}) \in \ker(A_F ^t - I) = \im g$. Let $a \in K_n (L_{\mathsf{k}} (F))$ be an element that $g$ maps to $\iota^{\reg} _{n-1} (\mathbf{u})$. Then 
$$g_0 (\phi_* ^{-1} (a)) = \pi_{n-1} ^\reg \circ g \circ \phi_* (\phi_* ^{-1} (a)) = \pi_{n-1} ^\reg \circ g(a) = \pi_{n-1} ^\reg \circ \iota^{\reg} _{n-1} (\mathbf{u})=\mathbf{u},$$ so $\mathbf{u} \in \im g_0.$ Thus $\ker \left( \begin{smallmatrix} B_E ^t - I \\ C_E ^t \end{smallmatrix} \right) \subseteq \im g_0$, and \eqref{main} is exact. \end{proof}

The following lemma regarding negative $K$-theory is included for the convenience of the reader. The proof requires putting together a few results in \cite{rosen}.
\begin{lemma}\label{negk}
If $R$ is a left regular ring and $n \geq 1$, then $K_{-n} (R) = \{ 0 \}$. In particular, if $\mathsf{k}$ is a field and $n \geq 1$, then $K_{-n} (\mathsf{k}) = \{0\}$.
\end{lemma}
\begin{proof}
The lemma is stated as fact in \cite[Definition~3.3.1]{rosen}, referring to \cite[Corollary~3.2.20]{rosen} and \cite[Theorem~3.2.3]{rosen}, but \cite[Corollary~3.2.13]{rosen} is also needed. For the last claim, note that a field is trivially a left regular ring.
\end{proof}

Combining Theorem~\ref{abc} with facts about the algebraic $K$-theory of fields, we obtain the following proposition,  which is well known in some special cases (e.g., \cite[Proposition~5.1]{rt}).

\begin{remark}
Let $G$ be a group and let $A$ be an integer matrix.  When $G$ is written additively, matrix multiplication $A : G^n \to G^n$ is done in the usual way, adding integer sums of elements of $G$.  When $G$ is written multiplicatively (such as in Proposition~\ref{3prop}(iii) when we identify $K_1(\mathsf{k})$ with $\mathsf{k}^\times$) we must take products of integer powers of $G$; for example, if we have $\left( \begin{smallmatrix} a & b \\ c & d \end{smallmatrix} \right) : (\mathsf{k}^\times ) ^2 \to (\mathsf{k}^\times)^2$, then this map takes $\left( \begin{smallmatrix} x \\ y \end{smallmatrix} \right) \in (\mathsf{k}^\times)^2$ to $\left( \begin{smallmatrix} x^a b^y \\ x^c y^d \end{smallmatrix} \right) \in (\mathsf{k}^\times)^2$.
\end{remark}

\begin{proposition} \label{3prop}
Let $E$ be a graph, let $\mathsf{k}$ be a field, and consider the Leavitt path algebra $L_\mathsf{k} (E)$. 
\begin{itemize}
\item[(i)]If $n \leq -1$, then $K_n (L_\mathsf{k} (E)) = \{0\}.$
\item[(ii)]$K_0 (L_\mathsf{k} (E)) \cong \coker \left( \left( \begin{matrix} B_E ^t - I \\ C_E ^t \end{matrix} \right): \Z^{E^0 _{\reg}} \to \Z^{E^0}\right).$
\item[(iii)]$K_1 (L_\mathsf{k} (E))$ is isomorphic to a direct sum: 
\begin{align*}K_1 (L_\mathsf{k} (E))  \cong \coker & \left( \left( \begin{matrix} B_E ^t - I \\ C_E ^t \end{matrix} \right): (\mathsf{k}^{\times})^{E^0 _{\reg}} \to (\mathsf{k}^{\times})^{E^0} \right) \\ & \oplus~
\ker \left( \left( \begin{matrix} B_E ^t - I \\ C_E ^t \end{matrix} \right): \Z^{E^0 _{\reg}} \to  \Z^{E^0} \right). \end{align*}
\end{itemize}
\end{proposition}

\begin{proof}
By Lemma~\ref{negk}, if $\mathsf{k}$ is any field and $n \leq -1$, then $K_n (\mathsf{k}) = \{0\}$. Combining this with Theorem~\ref{abc}, item (i) follows with an exactness argument. By \cite[Example~1.1.6]{rosen}, if $\mathsf{k}$ is any field, then $K_0 (\mathsf{k}) = \Z$.  Thus, item (ii) follows with the exact sequence in Theorem~\ref{abc} and item (i).

By \cite[Example~III.1.1.2]{kbook} or by \cite[Proposition~2.2.2]{rosen}, if $\mathsf{k}$ is any field, then $K_1 (\mathsf{k}) = \mathsf{k}^{\times}$. So at position $n=1$, the exact sequence of Theorem~\ref{abc} takes the form
$$\xymatrix{ \cdots \ar[r]
& (\mathsf{k}^{\times})^{E^0 _{\reg}} \ar[r]^-{\left( \begin{smallmatrix} B_E ^t - I \\ C_E ^t \end{smallmatrix} \right)}
& (\mathsf{k}^{\times})^{E^0} \ar[r]
& K_1 (L_\mathsf{k} (E)) \ar[r]
& \Z ^{E^0 _{\reg}} \ar[r]^-{\left( \begin{smallmatrix} B_E ^t - I \\ C_E ^t \end{smallmatrix} \right)}
& \cdots \\ },$$
which induces the short exact sequence
$$\xymatrix{0 \ar[r] & \coker {\begin{pmatrix} B_E ^t - I \\ C_E ^t \end{pmatrix}} \ar[r] & K_1 (L_\mathsf{k} (E)) \ar[r] & \ker {\begin{pmatrix} B_E^t - I \\ C_E ^t \end{pmatrix}} \ar[r] & 0}.$$
Since $\ker \left( \begin{smallmatrix} B_E ^t - I \\ C_E ^t \end{smallmatrix} \right)$ is a subgroup of the free abelian group $\Z^{E^0 _\reg }$, it follows that $\ker \left( \begin{smallmatrix} B_E ^t - I \\ C_E ^t \end{smallmatrix} \right) $ is a free abelian group. Hence the short exact sequence splits and item (iii) holds.
\end{proof}

\begin{theorem} \label{finite-vertices-computation}
Let $\mathsf{k}$ be a field and $n \in \N$. If $E$ is a graph with only finitely many vertices, then there exist $d_1, \ldots, d_k \in \{2, 3, 4, \ldots \}$ and $m \in \{0, 1, 2, \ldots \}$ such that $d_i | d_{i+1}$ for $1 \leq i \leq k-1$ and
\begin{align*} &\coker \left( \begin{pmatrix} B_E ^t - I \\ C_E ^t \end{pmatrix} : K_n (\mathsf{k})^{E^0 _\reg} \to K_n (\mathsf{k})^{E^0} \right) \\ 
& \qquad  \cong \bigslant{K_n (\mathsf{k})}{ \langle d_1 x : x \in K_n (\mathsf{k}) \rangle } \oplus \cdots \oplus \bigslant{K_n (\mathsf{k})}{ \langle d_k x : x \in K_n (\mathsf{k}) \rangle } \oplus K_n(\mathsf{k})^{m+|E^0 _\sing|} \end{align*}
and
\begin{align*} &\ker \left( \begin{pmatrix} B_E ^t - I \\ C_E ^t \end{pmatrix} : K_n (\mathsf{k})^{E^0 _\reg} \to K_n (\mathsf{k})^{E^0} \right) \\ 
& \qquad \qquad \qquad \qquad \cong K_n (\mathsf{k})^m \oplus \left( \bigoplus_{i=1} ^k \ker ((d_i):K_n (\mathsf{k}) \to K_n (\mathsf{k})) \right). \end{align*}
Moreover,
$$K_0 (L_\mathsf{k} (E)) \cong \Z_{d_1} \oplus \cdots \oplus \Z_{d_k} \oplus \Z^{m+|E^0 _\sing|}$$
and
$$K_1 (L_\mathsf{k} (E)) \cong  \bigslant{ \mathsf{k}^\times } {\{ x^{d_1} : x \in \mathsf{k}^\times \} } \oplus \cdots \oplus \bigslant{ \mathsf{k}^\times } { \{ x^{d_k} : x \in \mathsf{k}^\times \} } \oplus (\mathsf{k}^\times)^{m+|E^0 _\sing|}  \oplus \Z^m.$$
\end{theorem}
\begin{proof}
If $|E^0| < \infty$, then the matrix $\begin{pmatrix} B_E ^t -I \\ C_E ^t \end{pmatrix}$ has Smith normal form $\begin{pmatrix} D \\ 0 \end{pmatrix}$, where $0$ is the $|E^0 _\sing| \times |E^0 _\reg|$ matrix with $0$ in each entry and $D$ is the $|E^0 _\reg| \times |E^0 _\reg|$ diagonal matrix $\diag(1, \ldots, 1, d_1, \ldots, d_k, 0, \ldots, 0)$, with $0$ in each of the last $m$ entries for some $m \in \{0,1,2,\ldots \}$. If $\mathsf{k}$ is a field, then
\begin{align*} \coker & \left( \begin{pmatrix} B_E ^t - I \\ C_E ^t \end{pmatrix} : K_n (\mathsf{k})^{E^0 _\reg} \to K_n (\mathsf{k})^{E^0} \right) \cong \coker \left( \begin{pmatrix} D \\ 0 \end{pmatrix} : K_n (\mathsf{k})^{E^0 _\reg} \to K_n (\mathsf{k})^{E^0} \right) \\ & \cong \bigslant{K_n (\mathsf{k})} { \langle d_1 x : x \in K_n (\mathsf{k}) \rangle } \oplus \cdots \oplus \bigslant{ K_n (\mathsf{k}) } { \langle d_k x : x \in K_n (\mathsf{k}) \rangle } \oplus K_n(\mathsf{k})^{m+|E^0 _\sing|} \end{align*}
and
\begin{align*} \ker \left( \begin{pmatrix} B_E ^t - I \\ C_E ^t \end{pmatrix} : K_n (\mathsf{k})^{E^0 _\reg} \to K_n (\mathsf{k})^{E^0} \right) & \cong \ker \left( \begin{pmatrix} D \\ 0 \end{pmatrix} : K_n (\mathsf{k})^{E^0 _\reg} \to K_n (\mathsf{k})^{E^0} \right) \\  \cong  K_n & (\mathsf{k})^m \oplus  \bigoplus_{i=1} ^k \ker \left((d_i): K_n (\mathsf{k}) \to K_n (\mathsf{k}) \right). \end{align*}
When $n=0$ or $n=1$, then the last claim follows from Proposition~\ref{3prop} by substituting $\left( \begin{smallmatrix} D \\ 0 \end{smallmatrix} \right)$ for $\left(\begin{smallmatrix} B_E ^t - I \\ C_E ^t \end{smallmatrix}\right)$ in Proposition~\ref{3prop} (ii) and (iii).
\end{proof}

\begin{remark}
In Theorem~\ref{finite-vertices-computation}, the case when $k=0$ (and the list $d_1 , \ldots , d_k $ is empty) is possible.
\end{remark}

\section{$K$-theory for Leavitt Path Algebras over Finite Fields}

In this section we compute the $K$-groups of a Leavitt path algebra over a finite field $\mathsf{k}=\F_q$ with $q$ elements. (So $q=p^k$ for some prime $p$, where $p$ is the characteristic of the field.)

\begin{proposition} \label{LPA-K-theory-finite-field-even}
Let $E$ be a graph, let $\F_q$ be a finite field with $q$ elements, and consider the Leavitt path algebra $L_{\F_q} (E)$. If $n \geq 2$ is even and $n=2j$ for some $j \in \N$, then
$$K_n (L_{\F_q} (E)) \cong \ker \left( \left( \begin{matrix} B_E ^t - I \\ C_E ^t \end{matrix} \right):
\Z_{q^j -1} ^{E^0 _{\reg}} \to \Z_{q^j -1} ^{E^0} \right).$$
\end{proposition}
\begin{proof}
By \cite[Theorem~8(i)]{quillen},
if $j \geq 1$, then $K_{2j} (\F_q)=\{0\}$, and if $j \geq 2$, then $K_{2j-1}(\F_q) = \Z_{q^j - 1}$.
 Hence the exact sequence of Theorem~\ref{abc} takes the form
$$ \xymatrix{ \cdots \ar[r]^-{\left( \begin{smallmatrix} B_E ^t - I \\ C_E ^t \end{smallmatrix} \right)} & \{0\}^{E^0} \ar[r]^-{\phi} & K_n (L_{\F_q} (E)) \ar[r]^-{\psi} & \Z_{q^j -1}^{E^0 _{\reg}} \ar[r]^-{\left( \begin{smallmatrix} B_E ^t - I \\ C_E ^t \end{smallmatrix} \right)} &\Z_{q^j -1} ^{E^0} \ar[r] & \cdots \\},$$
where $n=2j.$
By exactness, $\psi$ is injective and
$$K_n (L_{\F_q} (E)) \cong \im \psi = \ker  \left( \begin{matrix} B_E ^t - I \\ C_E ^t \end{matrix} \right) .$$
\end{proof}

\begin{proposition} \label{LPA-K-theory-finite-field-odd}
Let $E$ be a graph, let $\F_q$ be a finite field with $q$ elements, and consider the Leavitt path algebra $L_{\F_q} (E)$. If $n \geq 3$ is odd and $n=2j-1$ for some $j \in \N$, then
$$K_n (L_{\F_q}(E)) \cong \coker\left(\left( \begin{matrix} B_E ^t - I \\ C_E ^t \end{matrix} \right): \Z_{q^j -1}^{E^0 _{\reg}} \to  \Z_{q^j -1} ^{E^0} \right).$$
\end{proposition}
\begin{proof}
By \cite[Theorem~8(i)]{quillen}, if $j \geq 1$, then $K_{2j} (\F_q)=\{0\}$, and if $j \geq 2$, then $K_{2j-1}(\F_q) = \Z_{q^j - 1}.$ So the exact sequence of Theorem~\ref{abc} becomes
$$\xymatrix{ \cdots \ar[r] &\Z_{q^j -1} ^{E^0 _{\reg}} \ar[r]^-{\left( \begin{smallmatrix} B_E ^t - I \\ C_E ^t \end{smallmatrix} \right)} & \Z_{q^j -1} ^{E^0} \ar[r]^-{\phi}& K_n (L_{\F_q} (E)) \ar[r]^-{\psi} & \{0\}^{E^0 _{\reg}} \ar[r] & \cdots \\},$$
where $n=2j-1.$
By exactness, $K_n (L_{\F_q} (E)) = \im \phi \cong \coker\left( \begin{smallmatrix} B_E ^t - I \\ C_E ^t \end{smallmatrix} \right).$
\end{proof}

\begin{theorem} \label{LPA-K-theory-finite-field}
Let $E$ be a graph, let $\F_q$ be a finite field with $q$ elements, and consider the Leavitt path algebra $L_{\F_q} (E)$.  Then for any $n \in \Z$ we have
$$K_n (L_{\F_q} (E)) \cong \begin{cases} 0 & \text{if $n \leq -1$} \\ 
 \coker\left( \left( \begin{smallmatrix} B_E ^t - I \\ C_E ^t \end{smallmatrix} \right) : \Z^{E^0_\reg} \to \Z^{E^0} \right) & \text{if $n = 0$} \\ 
\coker \left( \left( \begin{smallmatrix} B_E ^t - I \\ C_E ^t \end{smallmatrix} \right) : \Z_{q-1}^{E^0_\reg} \to \Z_{q-1}^{E^0} \right)  \\ \quad \oplus \ker \left( \left( \begin{smallmatrix} B_E ^t - I \\ C_E ^t \end{smallmatrix} \right) : \Z^{E^0_\reg} \to \Z^{E^0} \right) & \text{if $n = 1$} \\  
\ker \left( \left( \begin{smallmatrix} B_E ^t - I \\ C_E ^t \end{smallmatrix} \right):
\Z_{q^j -1} ^{E^0 _{\reg}} \to \Z_{q^j -1} ^{E^0} \right) & \text{if $n \geq 2$ is even and $n = 2j$}  \\
 \coker\left(\left( \begin{smallmatrix} B_E ^t - I \\ C_E ^t \end{smallmatrix} \right): \Z_{q^j -1}^{E^0 _{\reg}} \to  \Z_{q^j -1} ^{E^0}  \right) & \text{if $n \geq 3$ is odd and $n = 2j-1$.} 
\end{cases}$$
\end{theorem}

\begin{proof}
The case when $n \geq 2$ is even follows from Proposition~\ref{LPA-K-theory-finite-field-even}, and the case when $n \geq 3$ is odd follows from Proposition~\ref{LPA-K-theory-finite-field-odd}.  The case when $n \leq -1$  follows from Proposition~\ref{3prop}(i), and the case when $n=0$ follows from Proposition~\ref{3prop}(ii).  When $n=1$, Proposition~\ref{3prop}(iii) shows that 
$$K_1 (L_\mathsf{k} (E))  \cong \coker  \left( \left( \begin{smallmatrix} B_E ^t - I \\ C_E ^t \end{smallmatrix} \right): (\F_q^{\times})^{E^0 _{\reg}} \to (\F_q^{\times})^{E^0} \right)  \oplus 
\ker \left( \left( \begin{smallmatrix} B_E ^t - I \\ C_E ^t \end{smallmatrix} \right): \Z^{E^0 _{\reg}} \to  \Z^{E^0} \right).$$ Since $\F_q$ is a finite field with $q$ elements, if follows that the multiplicative group $\F_q^\times$ is a cyclic group of order $q-1$ (see, for example, \cite[Theorem~5.3]{Lang} or \cite[Theorem~22.2]{Gallian}).  Thus the multiplicative group $\F_q^\times$ is isomorphic to the additive group $\Z_{q-1}$, and there is an element $\alpha \in \F_q^\times$ of multiplicative order $q-1$ with the isomorphism from $\F_q^\times$ to $\Z_{q-1}$ given by $\alpha^n \mapsto n$.  Thus if $x_1, \ldots, x_k \in \F_q^\times$ with $x_i = \alpha^{n_i}$, then the isomorphism takes an element of the form $x_1^{d_1} \ldots x_k^{d_k} \in \F_q^\times$ to the element $d_1n_1 + \ldots + d_kn_k$.  It follows that $\coker  \left( \left( \begin{smallmatrix} B_E ^t - I \\ C_E ^t \end{smallmatrix} \right): (\F_q^{\times})^{E^0 _{\reg}} \to (\F_q^{\times})^{E^0} \right)$ (where the groups are written multiplicatively) is isomorphic to $\coker  \left( \left( \begin{smallmatrix} B_E ^t - I \\ C_E ^t \end{smallmatrix} \right): \Z_{q-1}^{E^0 _{\reg}} \to \Z_{q-1}^{E^0} \right)$ (where the groups are written additively).
\end{proof}

\section{Computations of $K$-theory for Certain Leavitt Path Algebras}

In this section we compute the $K$-groups of a Leavitt path algebra under certain hypotheses on the $K$-groups of the underlying field.  This allows us to calculate the $K$-groups of a Leavitt path algebra over any algebraically closed field.

\begin{lemma} \label{divisible-injective-direct-sum-lem}
If $G$ is an abelian group and $D$ is a divisible subgroup of $G$, then $G \cong D \oplus G/D$.
\end{lemma}

\begin{proof}
This follows from \cite[Ch.IV \S3 Lemma~3.9 and Proposition~3.13]{Hungerford}.
\end{proof}

$ $

\begin{theorem} \label{algclosed}
Let $E$ be a graph, let $\mathsf{k}$ be a field, and consider the Leavitt path algebra $L_\mathsf{k} (E)$. For each $n \in \Z$, if $K_n (\mathsf{k})$ is divisible or if $K_{n-1} (\mathsf{k})$ is free abelian, then $K_n (L_\mathsf{k} (E))$ is isomorphic to a direct sum: 
\begin{align*}K_n (L_\mathsf{k} (E)) \cong \coker & \left( \left( \begin{matrix} B_E ^t -I \\ C_E ^t \end{matrix} \right) : K_n (\mathsf{k})^{E_{\reg} ^0} \to K_n (\mathsf{k})^{E^0} \right) \\ & \oplus~ \ker \left( \left(\begin{matrix} B_E ^t -I \\ C_E ^t \end{matrix} \right) : K_{n-1} (\mathsf{k})^{E_{\reg} ^0} \to K_{n-1} (\mathsf{k})^{E^0} \right). \end{align*} In particular, these hypotheses are satisfied and the isomorphism holds for all $n \in \Z$ when $\mathsf{k}$ is an algebraically closed field. 
\end{theorem}
\begin{proof}
The long exact sequence of Theorem~\ref{abc} induces the short exact sequence
\begin{align*}
0 \longrightarrow \coker &\left( \left(  \begin{smallmatrix} B_E ^t - I \\ C_E ^t \end{smallmatrix} \right) : K_n (\mathsf{k})^{E^0_\textnormal{reg}} \to K_n (\mathsf{k})^{E^0} \right) \longrightarrow K_n (L_\mathsf{k}(E)) \\
 & \qquad \qquad  \longrightarrow  \ker \left( \left(  \begin{smallmatrix} B_E ^t - I \\ C_E ^t \end{smallmatrix} \right) : K_{n-1} (\mathsf{k})^{E^0_\textnormal{reg}} \to K_{n-1} (\mathsf{k})^{E^0} \right) \longrightarrow 0.
\end{align*} 
Suppose $K_n (\mathsf{k})$ is divisible. Since direct sums of divisible groups and quotients of divisible groups are divisible, the cokernel is divisible and Lemma~\ref{divisible-injective-direct-sum-lem} implies that
 the conclusion of the theorem holds. On the other hand, if $K_{n-1} (\mathsf{k})$ is free abelian, then $\ker \left( \begin{smallmatrix} B_E ^t - I \\ C_E ^t \end{smallmatrix} \right)$ is free abelian. This implies the short exact sequence splits, and the conclusion of the theorem holds.

If $\mathsf{k}$ is an algebraically closed field, then $K_n (\mathsf{k})$ is divisible for each $n \geq 1$. This follows from \cite[Theorem~VI.1.6]{kbook} if $\cha(\mathsf{k})=0$, and from \cite[Corollary~VI.1.3.1]{kbook} if $\cha(\mathsf{k}) \neq 0$. Moreover, if $n \leq 0$, then $K_{n-1} ((\mathsf{k}) = \{ 0 \}$ is free abelian.  Thus when $\mathsf{k}$ is an algebraically closed field, the hypotheses are satisfied and the isomorphism holds for all $n \in \Z^+$.
\end{proof}

\begin{example}
If $E$ is the graph consisting of one vertex and $n$ edges
$$
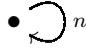
\xymatrix{
\bullet \ar@(ru,rd) []^{n}
}$$
and $n\geq 2$, then $L_\mathsf{k} (E)$ is isomorphic to the Leavitt algebra $L_n$. If $n=1$, then $L_\mathsf{k} (E)$ is isomorphic to the Laurent polynomials $\mathsf{k}[x, x^{-1}]$. We consider $K_j (L_\mathsf{k} (E))$ when (i)~$\mathsf{k}$ is algebraically closed with characteristic $0$, (ii)~$\mathsf{k} = \R$, and (iii)~$\mathsf{k}$ has characteristic $p>0$ and $n=p^m +1$.

First suppose $\mathsf{k}$ is an algebraically closed field with characteristic $0$. By \cite[Theorem~VI.1.6]{kbook}, $K_{2j}(\mathsf{k})$ is isomorphic to a uniquely divisible group, and $K_{2j-1}(\mathsf{k})$ is isomorphic to the direct sum of a uniquely divisible group and $\Q / \Z$, for $j \geq 1$. Here $\left( \begin{smallmatrix} B_E ^t - I \\ C_E^t \end{smallmatrix} \right)= (n-1)$. Multiplication by a nonzero integer is an isomorphism on a uniquely divisible group, so if $n \neq 1$, Theorem~\ref{algclosed} gives
\begin{align*}
K_{2j} &(L_\mathsf{k} (E)) \\
& \cong  \coker \left( (n-1) : K_{2j} (\mathsf{k}) \to K_{2j} (\mathsf{k}) \right)  \oplus~ \ker \left( (n-1) : K_{2j-1} (\mathsf{k}) \to K_{2j-1} (\mathsf{k}) \right) \\
& \cong \{ 0 \}  \oplus \ker( (n-1): \Q /\Z \to \Q / \Z) \cong \Z / (n-1)\Z \cong \Z_{n-1} \end{align*}
(where for the last isomorphism one checks that the class in $\Q/\Z$ represented by $\frac{1}{n-1}$ generates $\ker( (n-1): \Q /\Z \to \Q / \Z)$ with order $n-1$), and 
\begin{align*}
K_{2j+1} &(L_\mathsf{k} (E)) \\
& \cong  \coker \left( (n-1) : K_{2j+1} (\mathsf{k}) \to K_{2j+1} (\mathsf{k}) \right)  \oplus~ \ker \left( (n-1) : K_{2j} (\mathsf{k}) \to K_{2j} (\mathsf{k}) \right) \\
& \cong \coker ((n-1): \Q/\Z \to \Q / \Z) \oplus \{0 \} \cong \{0\}
\end{align*}
If $n=1$, then $K_j (L_\mathsf{k} (E)) \cong K_j (\mathsf{k}) \oplus K_{j-1} (\mathsf{k})$ for all $j \in \Z$.

Now suppose $\mathsf{k}=\R.$  By \cite[Theorem VI.3.1]{kbook} or \cite[Corollary 22.6]{grayson}, if $j \in \N$ then
$$K_j (\R) \cong \left\{\begin{array}{ll}
D_j \oplus \Z_2 & \text{if } j \equiv 1, 2 ~ (\modulo 8)\\
D_j \oplus \Q / \Z & \text{if } j \equiv 3,7 ~ (\modulo 8)\\
D_j & \text{if } j \equiv 0,4,5,6 ~ (\modulo 8),
\end{array} \right.$$
where $D_j$ is a uniquely divisible group. If $n$ is even, then
$$K_j (L_\R (E)) \cong \left\{\begin{array}{ll} 
0 &\text{if } j \equiv 1,2, 3 ~ (\modulo 4) \\
\Z_{n-1} & \text{if } j \equiv 0 ~ (\modulo 4),
\end{array} \right.$$
and if $n\neq 1$ is odd, then
$$K_j (L_\R (E)) \cong \left\{\begin{array}{ll} 
0 & \text{if } j \equiv 5,6,7 ~ (\modulo 8)\\
\Z_2 &\text{if } j \equiv 1, 3 ~ (\modulo 8)\\
\Z_{n-1} & \text{if } j \equiv 0, 4 ~ (\modulo 8) \\
\Z_4 \text{ or } \Z_2 \oplus \Z_2 & \text{if } j \equiv 2 ~ (\modulo 8)
\end{array} \right.$$
for $j \in \N$.

Finally, suppose $\mathsf{k}$ is a field with characteristic $p > 0$. By \cite[Theorem VI.4.7 (b)]{kbook}, it follows $K_j (\mathsf{k})$ has no $p$-torsion for $j \geq 0$. If $n=p^m + 1$ for some $m \in \N$, then using Theorem~\ref{abc} and the fact that $\ker((p^m): K_{j-1} (\mathsf{k}) \to K_{j-1} (\mathsf{k}))=0$, we obtain
$$K_j (L_\mathsf{k} (E)) = \coker((p^m): K_j (\mathsf{k}) \to K_j ( \mathsf{k})) \cong K_j (\mathsf{k}) / p^m K_j (\mathsf{k}).$$
\end{example}

\begin{definition}[The Cuntz Splice at a vertex $v$] \label{csdef} Let $E=(E^0, E^1, r_E, s_E)$ be a graph and let $v \in E^0$. Define a graph $F=(F^0, F^1, r_F, s_F)$ by $F^0 = E^0 \cup \{v_1, v_2 \}$, $F^1=E^1 \cup \{ e_1, e_2, f_1, f_2, h_1, h_2 \}$, and let $r_F$ and $s_F$ extend $r_E$ and $s_E$, respectively, and satisfy
$$s_F(e_1)=v, s_F(e_2)=v_1, s_F(f_1)=v_1, s_F (f_2) =v_2, s_F(h_1)=v_1, s_F (h_2)=v_2$$
and
$$r_F(e_1)=v_1, r_F(e_2)=v, r_F(f_1)=v_2, r_F(f_2)=v_1, r_F(h_1)=v_1, r_F(h_2)=v_2.$$
We say that $F$ is obtained by applying the Cuntz splice to $E$ at $v$. For example, the graph
$$
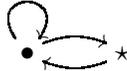
\xymatrix{ \bullet \ar@(lu,ru) [] \ar@/^/[r]& \star \ar@/^/[l]}$$
becomes
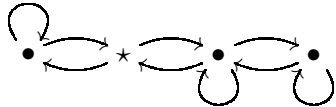
$$\xymatrix{ \bullet \ar@(lu,ru) [] \ar@/^/[r]& \star \ar@/^/[l] \ar@/^/[r] & \bullet \ar@(rd,ld) []\ar@/^/[r] \ar@/^/[l] & \bullet \ar@(rd,ld) [] \ar@/^/[l] }$$
$ $

\noindent if we apply the Cuntz splice at the $\star$ vertex.
\end{definition}
It was shown in \cite[Proposition 9.3]{rt} that if $E$ is a graph, then the Cuntz splice preserves the $K_0$-group and the $K_1$-group of the associated Leavitt path algebra $L_\mathsf{k} (E)$ for any choice of $\mathsf{k}$. Here we show that, for the kinds of fields described in Theorem~\ref{algclosed}, the Cuntz splice preserves the $K_n$-group of the associated Leavitt path algebra for all $n \in \Z$.

\begin{corollary}
Let $E$ be a graph, let $v\in E^0$, and let $F$ be a graph obtained by applying the Cuntz splice to $E$ at $v$. If $\mathsf{k}$ is a field and $n \in \Z$ such that either $K_n (\mathsf{k})$ is divisible or $K_{n-1} (\mathsf{k})$ is free, then $K_n (L_\mathsf{k} (E)) \cong K_n (L_\mathsf{k} (F))$.  
\end{corollary}
\begin{proof}
We use an argument similar to the one in \cite[Proposition 9.3(2)]{rt}. We begin by decomposing $E^0 = E_\reg ^0 \sqcup E_\sing ^0$ and writing the vertex matrix of $E$ as
$$A_E = \begin{pmatrix} B_E & C_E \\ * & * \end{pmatrix}.$$
Suppose $v$ is a regular vertex. Then the vertex matrix of $F$ has the form
$$A_F = \left(\begin{array}{cc|ccc|ccc}
1 & 1 & 0 & 0 & \cdots & 0 & 0 & \cdots \\ 
1 & 1 & 1 & 0 & \cdots & 0 & 0 & \cdots \\ \hline
0 & 1 & & & &&& \\ 
0 & 0 & & B_E& &&C_E& \\ 
\vdots & \vdots & & & &&& \\ \hline
0 & 0 & & & &&& \\ 
0 & 0 & & *& &&*& \\ 
\vdots & \vdots & & & &&& \\ 
\end{array} \right).$$
By Theorem~\ref{algclosed}, we obtain $K_n (L_\mathsf{k} (F))$ by considering the cokernel and kernel of the matrix
$$\left(\begin{array}{cc|ccc}
0 & 1 & 0 & 0 & \cdots \\
1 & 0 & 1 & 0 & \cdots \\ \hline
0 & 1 & && \\
0 & 0 & & B_E ^t - I & \\
\vdots & \vdots & & & \\ \hline
0 & 0 & && \\ 
0 & 0 & & C_E^t & \\
\vdots & \vdots & & 
\end{array}\right)
\begin{matrix} \text{ equivalent } \\ \longleftrightarrow \end{matrix}
\left(\begin{array}{cc|ccc}
1 & 0 & 0 & 0 & \cdots \\
0 & 1 & 0 & 0 & \cdots \\ \hline
0 & 0 & && \\
0 & 0 & & B_E ^t - I & \\
\vdots & \vdots & & & \\ \hline
0 & 0 & && \\ 
0 & 0 & & C_E^t & \\
\vdots & \vdots & & 
\end{array}\right).
$$
The $2 \times 2$ identity in the upper-left-hand corner has no effect on the cokernel and kernel, so
\begin{align*}K_n (L_\mathsf{k} (F)) \cong &\coker \left( \left( \begin{matrix} B_E ^t -I \\ C_E ^t \end{matrix} \right) : K_n (\mathsf{k})^{E_{\reg} ^0} \to K_n (\mathsf{k})^{E^0} \right) \\ & \qquad \oplus~  \ker \left( \left(\begin{matrix} B_E ^t -I \\ C_E ^t \end{matrix} \right) : K_{n-1} (\mathsf{k})^{E_{\reg} ^0} \to K_{n-1} (\mathsf{k})^{E^0} \right) \\
\cong & K_n (L_\mathsf{k}(E)). \end{align*}
Next, suppose $v$ is a singular vertex. Then the vertex matrix of $F$ has the form
$$A_F = \left(\begin{array}{cc|ccc|ccc}
1 & 1 & 0 & 0 & \cdots & 0 & 0 & \cdots \\ 
1 & 1 & 0 & 0 & \cdots & 1 & 0 & \cdots \\ \hline
0 & 0 & & & &&& \\ 
0 & 0 & & B_E& &&C_E& \\ 
\vdots & \vdots & & & &&& \\ \hline
0 & 1 & & & &&& \\ 
0 & 0 & & *& &&*& \\ 
\vdots & \vdots & & & &&& \\ 
\end{array} \right).$$
By Theorem~\ref{algclosed}, we obtain $K_n (L_\mathsf{k} (F))$ by considering the cokernel and kernel of the matrix
$$\left(\begin{array}{cc|ccc}
0 & 1 & 0 & 0 & \cdots \\
1 & 0 & 0 & 0 & \cdots \\ \hline
0 & 0 & && \\
0 & 0 & & B_E ^t - I & \\
\vdots & \vdots & & & \\ \hline
0 & 1 & && \\ 
0 & 0 & & C_E^t & \\
\vdots & \vdots & & 
\end{array}\right)
\begin{matrix} \text{ equivalent } \\ \longleftrightarrow \end{matrix}
\left(\begin{array}{cc|ccc}
1 & 0 & 0 & 0 & \cdots \\
0 & 1 & 0 & 0 & \cdots \\ \hline
0 & 0 & && \\
0 & 0 & & B_E ^t - I & \\
\vdots & \vdots & & & \\ \hline
0 & 0 & && \\ 
0 & 0 & & C_E^t & \\
\vdots & \vdots & & 
\end{array}\right).
$$
The $2 \times 2$ identity in the upper-left-hand corner has no effect on the cokernel and kernel, so as above
$K_n (L_\mathsf{k} (F)) \cong  K_n (L_\mathsf{k}(E)).$
\end{proof}

The following theorem is inspired by \cite[Theorem 9.4]{abc}.

\begin{theorem}
Let $E$ be a finite graph with no sinks such that $\det(A_E ^t - I) \neq 0$. If $\mathsf{k}$ is an algebraically closed field, then

$$K_n (L_\mathsf{k} (E)) \cong \left\{ \begin{array}{ll}
0 & \text{if } n \geq 1 \text{ is odd} \\
\ker((A_E ^t - I): G^{E^0} \to G^{E^0}) & \text{if } n \geq 2 \text{ is even} \\
\coker((A_E ^t - I): \Z^{E^0} \to \Z^{E^0})  &  \text{if } n = 0,
\end{array}\right.$$
where $G:= \Q / \Z$ if $\cha(\mathsf{k})=0$ or $G:= \Q / \Z[\frac{1}{p}]$ if $\cha (\mathsf{k})=p>0$.

Note that this produces a weak ``Bott periodicity" for Leavitt path algebras over algebraically closed fields and with $\det(A_E ^t - I) \neq 0$: Under these hypotheses we have that $K_{2n } ( L_{\mathsf{k}} (E) ) \cong K_{2}( L_{\mathsf{k}} (E) )$ and $K_{2n-1  } ( L_{\mathsf{k}} (E) ) \cong K_{1}( L_{\mathsf{k}} (E) ) \cong 0$ for all $n \in \N$.  Moreover,  if $\cha(\mathsf{k}) \nmid \det(A_E ^t - I)$, then we also have $K_0 ( L_{\mathsf{k}} (E) ) \cong K_{2}( L_{\mathsf{k}} (E) )$, so that $K_{2n } ( L_{\mathsf{k}} (E) ) \cong K_{0}( L_{\mathsf{k}} (E) )$ for all $n \in \Z^+$
\end{theorem}
\begin{proof}
We may write
$$K_n (\mathsf{k}) \cong \left\{ \begin{array}{ll}
D_n \oplus G & \text{if } n \geq 1 \text{ is odd} \\
D_n & \text{if } n \geq 2 \text{ is even} \\
 \Z  &  \text{if } n = 0,
\end{array}\right.$$
where $D_n$ is a uniquely divisible group, $G= \Q / \Z$ if char($\mathsf{k}$) $=0$ by \cite[Theorem~VI.1.6]{kbook}, and $G= \Q / \Z[\frac{1}{p}]$ if char($\mathsf{k}$) $=p>0$ by \cite[Corollary~VI.1.3.1]{kbook}. In either case $G$ is divisible.

Since $(A_E ^t - I)$ is an $E^0 \times E^0$ matrix with nonzero determinant, it has Smith normal form with nonzero diagonal entries and zeros elsewhere. Since $D_n$ is uniquely divisible, $(A_E ^t - I):D_n ^{E^0} \to D_n ^{E^0}$ is an isomorphism, and since $G$ is divisible, $(A_E ^t - I): G^{E^0} \to G^{E^0}$ is a surjection.
If $n$ is odd, the map
$$(A_E ^t - I) : K_n (\mathsf{k}) ^{E^0} \to K_n (\mathsf{k}) ^{E^0}$$
decomposes as
$$(A_E ^t - I) \oplus (A_E ^t -I) : D_n ^{E^0} \oplus G^{E^0} \to D_n ^{E^0} \oplus G^{E^0},$$
and is an isomorphism in the first summand and a surjection in the second.
Thus, when we apply Theorem~\ref{algclosed}, for odd $n \geq 1$ we obtain $K_n (L_\mathsf{k} (E))  \cong \{0\} \oplus \{0\}$, and for even $n \geq 2$ we obtain $K_n (L_\mathsf{k} (E))  \cong \{ 0 \} \oplus \ker\left((A_E^t -I): G^{E^0} \to G^{E^0}\right)$.  In addition, for $n=0$ we have $K_0 (L_\mathsf{k} (E))  \cong  \coker((A_E ^t - I): \Z^{E^0} \to \Z^{E^0})$ by Proposition~\ref{3prop}.

Finally, suppose that the additional hypothesis $\cha(\mathsf{k}) \nmid \det(A_E ^t - I)$ holds.  Let $\mathrm{diag} ( n_{1} , \dots, n_{s} )$ be the Smith normal form for the matrix $A_{E}^{t} - I$.  Then $\det(A_E ^t - I) = n_1 \ldots n_s$, and since $\cha(\mathsf{k}) \nmid \det(A_E ^t - I)$, we may conclude that $\cha(\mathsf{k}) \nmid n_{i}$ for all $1 \leq i \leq s$.  If we consider the multiplication $(n_i) : G \to G$, a straightforward computation shows that this map has kernel equal to $\left\langle \overline{ \frac{1}{n_{i} } } \right\rangle$, where $\overline{\frac{1}{n_{i}}}$ is the class in $G$ represented by the element $\frac{1}{n_{i}}$ and $\left\langle \overline{ \frac{1}{n_{i} } } \right\rangle$ is the additive subgroup generated by $\overline{\frac{1}{n_{i}}}$.  Since this is a cyclic group of order $n_i$, we have that $(\ker (n_i) : G \to G) \cong \Z_{n_i}$, and therefore, 
\[
\ker((A_E ^t - I): G^{E^0} \to G^{E^0}) \cong \Z_{n_{1}} \oplus \cdots \oplus \Z_{n_{s}} \cong  \coker((A_E ^t - I): \Z^{E^0} \to \Z^{E^0}).
\]
\end{proof}

\section{Rank and Corank of an Abelian Group}

In this section we examine and develop basic properties of the rank and corank of an abelian group.  We also compare the values that the rank and corank can assign to an abelian group.

\subsection{Rank of an Abelian Group}

\begin{definition} \label{rank-def}
If $(G,+)$ is an abelian group, a finite collection of elements $\{ g_i \}_{i=1}^k \subseteq G$ is \emph{linearly independent (over $\Z$)} if whenever $\sum_{i=1}^k n_i g_i = 0$ for $n_1, \ldots, n_k \in \Z$, then $n_1 = \ldots = n_k = 0$.  Any two maximal linearly independent sets in $G$ have the same cardinality, and we define $\rank G$ to be this cardinality if a maximal linearly independent set exists and $\infty$ otherwise.  
\end{definition}

\begin{remark}
Let $G$ be an abelian group.  One can see that if $G$ contains a linearly independent set with $n$ elements, then there exists an injective homomorphism $\iota : \Z^n \to G$ (given a linearly independent set of $n$ elements in $G$, the fact $\Z^n$ is free abelian gives a homomorphism taking the generators of $\Z^n$ to these elements and the linear independence implies this homomorphism is injective).  Conversely, any injective homomorphism $\iota : \Z^n \to G$ will send the generators of $\Z^n$ to a set of $n$ linearly independent elements in $G$.  Thus
\begin{equation} \label{rank-inj-eqn}
\rank G = \sup \{ n \in \Z^+ : \text{there exists an injective homomorphism $\iota:  \Z^n \to G$} \}.
\end{equation}
Furthermore, if we form the tensor product $\Q \otimes_\Z G$, then since $\Q$ is a field, $\Q \otimes_\Z G$ is a vector space, and maximal linearly independent sets in $G$ correspond to bases in $\Q \otimes_\Z G$.  Thus 
\begin{equation} \label{rank-def-eqn}
\rank G = \dim_\Q (\Q \otimes_\Z G)
\end{equation}
where $\dim_\Q$ denotes the dimension as a $\Q$-vector space.

The equations in \eqref{rank-inj-eqn} and \eqref{rank-def-eqn} give two equivalent ways to define the rank of an abelian group.
\end{remark}

\begin{remark}
It is important to notice that we are working in the category of abelian groups and defining the ``rank of an abelian group".  This is different from how the ``rank of a group" is defined: If $G$ is a (not necessarily abelian) group, then the rank of $G$ is defined to be the smallest cardinality of a generating set for $G$.  These notions do not coincide; for example the group-rank of $\Z_n$ is 1, while using the abelian-group-rank from Definition~\ref{rank-def} we have $\rank \Z_n = 0$.  Sometime the term ``torsion-free rank" or ``Pr\"ufer rank" is used for this abelian-group-rank; however, we are going to simply call it ``rank" with the understanding we are working in the category of abelian groups.
\end{remark}

The following are some well-known facts about the torsion-free rank of an abelian group.  We will use these facts in the next sections.

\begin{proposition} \label{rank-elem-facts-prop}
The rank of an abelian group satisfies the following elementary properties:
\begin{enumerate}
\item[(i)] If $G$ and $H$ are isomorphic abelian groups, then $\rank G = \rank H$.
\item[(ii)]  $\rank \Z = 1$
\item[(iii)]  $\rank G = 0$ if and only if $G$ is a torsion group
\item[(iv)]  If $0 \to P \to Q \to R \to 0$ is an exact sequence of abelian groups $P$, $Q$, and $R$, then $\rank Q = \rank  P + \rank R$. 
\item[(v)]  If $G_1$ and $G_2$ are abelian groups, then $\rank (G_1 \oplus G_2) = \rank G_1 + \rank G_2$
\item[(vi)] If $\rank G = n < \infty$, then there exists an injective homomorphism $\iota :  \Z^n \to G$, and if $\rank G = \infty$ there exists an injective homomorphism $\iota : \bigoplus_{i=1}^\infty \Z \to G$.
\end{enumerate}
\end{proposition}

\begin{proof}
Items (i)--(v) follow from well-known properties of vector spaces over a field and the fact that $\Q$ is a flat $\Z$-module.  For (vi), if $\rank G < \infty$ the result follows from \eqref{rank-inj-eqn}. If $\rank G = \infty$, then \eqref{rank-def-eqn} implies $\Q \otimes_\Z G$ is an infinite-dimensional vector space over $\Q$ and hence contains an infinite basis, which corresponds to an infinite set in $G$ for which every finite subset is linearly independent.  The fact $\bigoplus_{i=1}^\infty \Z$ is a free abelian group with countably many generators implies there exists a homomorphism $\iota : \bigoplus_{i=1}^\infty \Z \to G$ and the fact every finite subset is linearly independent implies $\iota$ is injective.

\end{proof}

\subsection{Corank of an Abelian Group}

In analogy with the equation for the rank of an abelian group derived in \eqref{rank-inj-eqn}, we make the following definition.

\begin{definition} \label{corank-def}
If $G$ is an abelian group we define the \emph{corank of G} to be
$$\corank G := \sup \{ n \in \Z^+ : \text{there exists a surjective homomorphism $\pi : G \to \Z^n$} \}.$$
Note that $\corank G$ is an element of the extended positive integers $\{0, 1, 2, \ldots, \infty \}$.
\end{definition}

\begin{remark}
If $G$ is an abelian group and $\pi : G \to \Z^n$ is a surjective homomorphism, then $G / \ker \pi \cong \Z^n$.  Conversely, if $N$ is a subgroup of $G$ with $G/N \cong \Z^n$, then the quotient map from $G$ onto $G/N$ composed with an isomorphism from $G/N$ onto $\Z^n$ is surjective.  Hence
\begin{equation} \label{corank-free-quotient-eq}
\corank G = \sup \{ \rank (G/N) : N \text{ is a subgroup of $G$ and $G/N$ is free abelian} \}.
\end{equation}
\end{remark}

\begin{remark}
If $G$ is an abelian group and $\corank G = \infty$, then Definition~\ref{corank-def} implies that for every $n \in \Z$ there exists a surjective homomorphism from $G$ onto $\Z^n$.  However, if $\corank G = \infty$ it is not necessarily true that there exists a surjective homomorphism from $G$ onto $\bigoplus_{i=1}^\infty \Z$.  For example, if we consider the infinite direct product $\prod_{i=1}^\infty \Z$, then we see that projecting onto the first $n$ coordinates gives a surjective homomorphism onto $\Z^n$, and hence $\corank \prod_{i=1}^\infty \Z = \infty$.  However, there is no surjective homomorphism from $\prod_{i=1}^\infty \Z$ onto $\bigoplus_{i=1}^\infty \Z$ (The reason for this is that $\bigoplus_{i=1}^\infty \Z$ is a ``slender" group, see \cite[Chapter~VIII, \S94]{Fuc}.)  Contrast this situation with what occurs when $\rank G = \infty$ in Proposition~\ref{rank-elem-facts-prop}(vi).
\end{remark}

\begin{definition}
Recall that an abelian group $G$ is \emph{divisible} if for every $y \in G$ and for every $n \in \N$, there exists $x \in G$ such that $nx=y$.  Likewise, an abelian group $G$ is \emph{weakly divisible} if for every $y \in G$ and for every $N \in \N$, there exists $n \geq N$ and $x \in G$ such that $nx=y$.
\end{definition}

\begin{proposition} \label{corank-elem-facts-prop}
The corank of an abelian group satisfies the following elementary properties:
\begin{enumerate}
\item[(i)] If $G$ and $H$ are isomorphic abelian groups, then $\corank G = \corank H$.
\item[(ii)]  $\corank \Z^n = n$.
\item[(iii)]  If $G$ is a torsion group, then $\corank G = 0$.
\item[(iv)]  If $G$ is a weakly divisible group, then $\corank G = 0$.  
\item[(v)]  $\corank G = 0$ if and only if $G$ has no nonzero free abelian quotients.
\end{enumerate}
\end{proposition}

\begin{proof}
The fact in (i) follows immediately from the definition of $\corank$. For (ii) we see that the identity map is a surjective homomorphism from $\Z^n$ onto $\Z^n$ and observe that there are no surjective homomorphisms from $\Z^n$ onto $\Z^m$ when $m > n$.  For (iii) observe that if $G$ is a torsion group, then since homomorphisms take elements of finite order to elements of finite order, there are no nonzero homomorphisms from $G$ to $\Z^n$, and hence $\corank G = 0$.  For (iv) we observe that the homomorphic image of a weakly divisible group is weakly divisible, and since $\Z^n$ has no weakly divisible subgroups other than zero, $\corank G = 0$.  For (v) we see from \eqref{corank-free-quotient-eq} that $\corank G = 0$ if and only if every free abelian quotient of $G$ is the zero group.
\end{proof}

In addition to these facts, there are two important properties of $\corank$ that are not immediate that we establish in Proposition~\ref{quotient-zero-no-change-prop} and Proposition~\ref{corank-direct-sum-prop}.

\begin{proposition} \label{quotient-zero-no-change-prop}
If $G$ is an abelian group and $H$ is a subgroup of $G$ with $\corank H = 0$, then $\corank (G/H) = \corank G$.
\end{proposition}

\begin{proof}
Since the quotient homomorphism $q : G \to G/N$ is surjective, we see that any surjective homomorphism $\pi : G/N \to \Z^n$ may be precomposed with $q$ to obtain a surjective homomorphism $\pi \circ q : G \to \Z^n$.  Hence $\corank (G/H) \leq \corank G$.

To obtain the inequality in the other direction, suppose that $n \in \Z^+$ and $\pi : G \to \Z^n$ is a surjective homomorphism.  If $N := \ker \pi$, then $G/N \cong \Z^n$.  Since $(H + N)/N$ is a subgroup of $G/N$, and subgroups of free abelian groups are free abelian, $(H+N)/N$ is free abelian.  Since $(H+N)/N \cong H / (H \cap N)$, we have that $ H / (H \cap N)$ is a free abelian group.  Since $\corank H = 0$, it follows from \eqref{corank-free-quotient-eq} that $H$ contains no nonzero free quotients.  Hence $H / (H \cap N) = 0$, and $H \cap N = H$, so that $H \subseteq N$.  Since $(G/H) / (N/H) \cong G/N \cong \Z^n$, there is a surjective homomorphism from $G/H$ onto $\Z^n$.  It follows that $\corank G \leq \corank (G/H)$.
\end{proof}

\begin{lemma} \label{corank-direct-sum-one-zero-lem}
If $G_1$ and $G_2$ are abelian groups and $\corank G_2 = 0$, then $\corank (G_1 \oplus G_2) = \corank G_1$.
\end{lemma}

\begin{proof}
Since $\corank (0 \oplus G_2) = \corank G_2 = 0$, Proposition~\ref{quotient-zero-no-change-prop} implies that $$\corank (G_1 \oplus G_2) = \corank ( (G_1 \oplus G_2) / (0 \oplus G_2) ) = \corank (G_1 \oplus 0) = \corank G_1.$$
\end{proof}

\begin{lemma} \label{split-off-Zn-corank-zero-lem}
If $\corank G = n < \infty$, then $G \cong \Z^n \oplus H$ for an abelian group $H$ with $\corank H = 0$.
\end{lemma}

\begin{proof}
By the definition of $\corank$ there exists a surjective homomorphism $\pi : G \to \Z^n$.  Let $H := \ker \pi$.  Then $G/H \cong \Z^n$.  Since $G/H$ is a free abelian group, the short exact sequence $0 \to H \to G \to G/H \to 0$ splits and $G \cong G/H \oplus H \cong \Z^n \oplus H$.

Since $G \cong \Z^n \oplus H$, there is a surjective homomorphism from $G$ onto $H$, and the fact that $\corank G < \infty$ implies $\corank H < \infty$.  Let $m = \corank H < \infty$.  Since $m$ is finite there exists a surjective homomorphism $\pi' : H \to \Z^m$.  As above, if we let $N := \ker \pi'$, then $H/N \cong \Z^m$, and since $H/N$ is a free abelian group, the short exact sequence $0 \to N \to H \to H/N \to 0$ splits and $H \cong H/N \oplus N \cong \Z^m \oplus N$.  Thus $G \cong  \Z^n \oplus H \cong \Z^n \oplus \Z^m \oplus N$, and there is a surjective homomorphism from $G$ onto $\Z^{n+m}$.  Hence $n + m \leq \corank G = n$, and since $m$ and $n$ are non-negative, we conclude that $m=0$.
\end{proof}

\begin{proposition} \label{corank-direct-sum-prop}
If $G_1$ and $G_2$ are abelian groups, then $$\corank (G_1 \oplus G_2) = \corank G_1 +  \corank G_2.$$
\end{proposition}

\begin{proof}
If $\corank G_1 = \infty$, then for every $n \in \N$ there exists a surjective homomorphism $\pi : G_1 \to \Z^n$.  If we precompose with the projection onto the first coordinate, $\pi_1 : G_1 \oplus G_2 \to G_1$ given by $\pi_1(g_1,g_2) := g_1$, then $\pi \circ \pi_1 : G_1 \oplus G_2 \to \Z^n$ is surjective.  It follows that $\corank (G_1 \oplus G_2) = \infty$.  Thus $\corank (G_1 \oplus G_2) = \infty = \infty+ \corank G_2 = \corank G_1 + \corank G_2$.

A similar argument shows that if $\corank G_2 = \infty$, then $\corank (G_1 \oplus G_2) = \infty$ and $\corank (G_1 \oplus G_2) = \corank G_1 + \corank G_2$.

If $\corank G_1 = n_1 < \infty$ and $\corank G_2 = n_2 < \infty$, then Lemma~\ref{split-off-Zn-corank-zero-lem} implies that $G_1 \cong \Z^{n_1} \oplus H_1$ and $G_2 \cong \Z^{n_2} \oplus H_2$ for some abelian groups $H_1$ and $H_2$ with $\corank H_1 = \corank H_2 = 0$.  Thus
\begin{align*}
\corank (G_1 \oplus G_2) &= \corank ((\Z^{n_1} \oplus H_1) \oplus (\Z^{n_2} \oplus H_2)) \\
&= \corank (\Z^{n_1}  \oplus \Z^{n_2} \oplus H_1 \oplus H_2) \\
&= \corank (\Z^{n_1}  \oplus \Z^{n_2} \oplus H_1) \qquad   \text{(by Lemma~\ref{corank-direct-sum-one-zero-lem})} \\
&= \corank (\Z^{n_1}  \oplus \Z^{n_2}) \qquad \qquad \text{(by Lemma~\ref{corank-direct-sum-one-zero-lem})} \\
&=n_1 + n_2 \\
&= \corank G_1 + \corank G_2.
\end{align*}
\end{proof}

\subsection{A Comparison of Rank and Corank}

We begin by computing the rank and corank of some groups to observe that their values do not always agree.

\begin{example} \label{compute-rank-corank-ex}
If $\Q$ denotes the abelian group of rational numbers with addition, then we see that $\corank \Q = 0$ since $\Q$ is divisible.  In addition, $\rank \Q \geq 1$ since $\Q$ is a not a torsion group, and for any set of two elements $\{ m, n \} \subseteq \Q$, we have $n(m) - m(n) = 0$ so that $\{ m, n \}$ is linearly dependent.  Hence $\rank \Q = 1$.

If $\R$ denotes the abelian group of real numbers with addition, then we see that $\corank \R = 0$ since $\R$ is divisible. For any prime $p$, the set of square roots of prime numbers up to $p$, namely $\{\sqrt{2}, \sqrt{3}, \sqrt{5}, \sqrt{7}, \sqrt{11}, \ldots, \sqrt{p}  \}$, is a linearly independent subset of $\R$, so $\rank \R = \infty$.

If we let $\prod_{i=1}^\infty \Z$ be the product of countably many copies of $\Z$, then for any $n \in \N$ there is an injection $\iota_n : \Z^n \to \prod_{i=1}^\infty \Z$ obtained by including into the first $n$ coordinates, and there is surjection $\pi_n : \prod_{i=1}^\infty \Z \to \Z^n$ obtained by projecting onto the first $n$ coordinates.  Hence $\rank \prod_{i=1}^\infty \Z = \corank \prod_{i=1}^\infty \Z = \infty$.

We display these results here for easy reference:
\begin{align*}
\corank \Q = 0  \qquad  \qquad \qquad & \rank \Q = 1 \\
\corank \R = 0  \qquad \qquad \qquad  & \rank \R = \infty \\
\corank \prod_{i=1}^\infty \Z = \infty  \qquad \qquad \qquad  & \rank \prod_{i=1}^\infty \Z = \infty
\end{align*}
\end{example}

\begin{example} \label{corank-not-exact-ex}
Let $\iota : \Z \to \Q$ be the inclusion map.  Note that $0 \to \Z \to \Q \to \Q/\Z \to 0$ is a short exact sequence.  However, $\corank \Z = 1$ and $\corank \Q = 0$, so that $\corank \Q \neq \corank \Z + \corank ( \Q / \Z)$.   Contrast this with the property of rank described in Proposition~\ref{rank-elem-facts-prop}(iv).
\end{example}

Although the above examples show that $\rank$ and $\corank$ do not agree in general, the following proposition shows that they do agree on finitely generated abelian groups and on free abelian groups.

\begin{proposition}
Let $G$ be an abelian group such that $G \cong T \oplus F$, where $T$ is a torsion group and $F$ is a free group.  Then $\rank G = \corank G = \rank F$. 
\end{proposition}

\begin{proof}
Proposition~\ref{rank-elem-facts-prop} implies that $\rank G = \rank T + \rank F = \rank F$.  Proposition~\ref{corank-direct-sum-prop} and Proposition~\ref{corank-elem-facts-prop} imply that $\corank G = \corank T + \corank F = \corank F$.  Since $F$ is a free abelian group, $F \cong \bigoplus_{i \in I} \Z$.  Thus $\rank F = |I|$ if $I$ is finite, and $\rank F = \infty$ if $I$ is infinite.  Likewise, $\corank F = |I|$ if $I$ is finite, and $\corank F = \infty$ if $I$ is infinite.  Hence, $\rank G = \rank F = \corank G$.
\end{proof}

The next proposition gives further insight into the relationship between $\rank$ and $\corank$ for general abelian groups.  It also shows that the problem of finding an abelian group with unequal rank and corank is tantamount to finding an abelian group with corank zero and nonzero rank.

\begin{proposition} \label{corank-smaller-than-rank-prop}
If $G$ is an abelian group, then $\corank G \leq \rank G$.  Furthermore, if $\rank G \neq \corank G$, then $\corank G < \infty$ and $G \cong \Z^n \oplus H$, where $n = \corank G$ and $H$ is an abelian group with $\corank H = 0$ and $\rank H = \rank G - n$.
\end{proposition}

\begin{proof}
For any $n \in \Z^+$, if there is a surjection $\pi : G \to \Z^n$, then $0 \to \ker \pi \to G \to \Z^n \to 0$ is a short exact sequence that splits (due to the fact $\Z^n$ is free) implying that $G \cong \Z^n \oplus \ker \pi$.  Hence, by Proposition~\ref{rank-elem-facts-prop}, $\rank G = n + \rank \ker \pi \geq n$.  This implies $\corank G \leq \rank G$.

If $\rank G \neq \corank G$, then the previous paragraph implies that $\corank G$ is finite.  Hence by Lemma~\ref{split-off-Zn-corank-zero-lem} we have $G \cong \Z^n \oplus H$ for an abelian group $H$ with $\corank H = 0$.  Thus $\rank G = n + \rank H$, and $\rank H = \rank G - n$.
\end{proof}

Proposition~\ref{corank-smaller-than-rank-prop} shows that to find abelian groups with unequal rank and corank, one needs to focus on finding abelian groups with zero corank and positive rank.

\begin{example}
If $0 \leq m < n \leq \infty$, we may define $G := \Z^m \oplus \Q^{n-m}$.  Then, using the computations from Example~\ref{compute-rank-corank-ex} we see that $\corank G = m + 0 = m$ and $\rank G = m + (n-m) = n$.  Thus for any $0 \leq m < n \leq \infty$ there exists an abelian group $G$ with $\corank G = m$ and $\rank G = n$.

Proposition~\ref{corank-smaller-than-rank-prop} shows we must have the corank of an abelian group less than or equal to the rank, but this example shows that for all values $0 \leq m < n \leq \infty$ there exists an abelian group $G$ with $\corank G = m$ and $\rank G = n$.  Moreover, Proposition~\ref{corank-smaller-than-rank-prop} implies that any such examples must be of the form $\Z^n \oplus H$ with $\corank H = 0$ and $\rank H \geq 1$.
\end{example}

\section{Size Functions on Abelian Groups}

\begin{definition} \label{size-function-def}
A \emph{size function} on the class of abelian groups is an assignment $$F : \Ab \to \Z^+ \cup \{ \infty \}$$ from the class of abelian groups $\Ab$ to the extended non-negative integers $\Z^+ \cup \{ \infty \} = \{0, 1, 2, \ldots, \infty \}$ satisfying the following conditions:
\begin{itemize}
\item[(1)] If $G_1$ and $G_2$ are abelian groups with $G_1 \cong G_2$, then $F(G_1) = F(G_2)$.
\item[(2)] If $G$ is a torsion group, then $F(G) = 0$.
\item[(3)] If $G$ is an abelian group and $H$ is a subgroup of $G$ with $F(H) = 0$, then $F(G/H) = F(G)$.
\item[(4)] If $G_1$ and $G_2$ are abelian groups, then $F(G_1 \oplus G_2) = F(G_1) + F(G_2)$.
\end{itemize}
\end{definition}

\begin{definition} \label{exact-size-function-def}
An \emph{exact size function} on the class of abelian groups is an assignment $$F : \Ab \to \Z^+ \cup \{ \infty \}$$ from the class of abelian groups $\Ab$ to the extended non-negative integers $\Z^+ \cup \{ \infty \} = \{0, 1, 2, \ldots, \infty \}$ satisfying the following conditions:
\begin{itemize}
\item[(1)] If $G_1$ and $G_2$ are abelian groups with $G_1 \cong G_2$, then $F(G_1) = F(G_2)$.
\item[(2)] If $G$ is a torsion group, then $F(G) = 0$.
\item[(3)] If $P$, $Q$, and $R$ are abelian groups and $0 \to P \to Q \to R \to 0$ is an exact sequence, then $F(Q) = F(P) + F(R)$.
\end{itemize}
\end{definition}

\begin{remark}
We point out that in both Definition~\ref{size-function-def} and Definition~\ref{exact-size-function-def} the domain of the assignment is a class (and not a set).  Thus, despite the name, size functions and exact size functions are technically assignments and not functions.
\end{remark}

Observe that \textit{a priori} properties (3) and (4) of Definition~\ref{size-function-def} are not required to hold for an exact size function.  The following proposition shows that, despite this, they do follow.

\begin{proposition} \label{exact-are-size-fcts-prop}
Any exact size function is also a size function.
\end{proposition}

\begin{proof}
It suffices to show that any exact size function satisfies properties (3) and (4) of Definition~\ref{size-function-def}.  To establish (3), let $G$ be an abelian group and let $H$ be a subgroup of $G$ with $F(H)= 0$.  Then $0 \to H \to G \to G/H \to 0$ is exact, and hence $F(G) = F(H) + F(G/H) = 0 + F(G/H) = F(G/H)$.  To establish (4), suppose $G_1$ and $G_2$ are abelian groups, and consider the exact sequence $0 \to G_1 \oplus 0 \to G_1 \oplus G_2 \to 0 \oplus G_2 \to 0$.  Then $F(G_1 \oplus G_2) = F(G_1 \oplus 0) + F(0 \oplus G_2) = F(G_1) + F(G_2)$.
\end{proof}

Although any exact size function is a size function, the converse does not hold (see Example~\ref{size-fct-not-exact-ex}).  However, any size function will satisfy the following special case of exactness.

\begin{lemma} \label{size-fct-limited-exactness-lem}
Let $F : \Ab \to \Z^+ \cup \{ \infty \}$ be a size function.  If $P$, $Q$, and $R$ are abelian groups, $0 \to P \to Q \to R \to 0$ is an exact sequence, and $F(P)=0$, then $F(Q) = F(R)$.
\end{lemma}

\begin{proof}
Due to the exactness of the sequence there is a subgroup $H$ of $Q$ such that $P \cong H$ and $Q/H \cong R$.  Thus, using the properties of a size function, we see that $F(H) = F(P) = 0$, and hence $F(R) = F(Q/H) = F(Q)$.
\end{proof}

The following proposition shows that on finitely generated abelian groups a size function is a constant multiple of the rank function.

\begin{proposition} \label{constant-multiple-rank-prop}
If $F : \Ab \to \Z^+ \cup \{ \infty \}$ is a size function, and $k := F(\Z)$, then $F(G) = k \rank G$ whenever $G$ is a finitely generated abelian group.
\end{proposition}

\begin{proof}
If $G$ is a finitely generated abelian group, then the fundamental theorem of finitely generated abelian groups implies $G \cong \Z^n \oplus T$ for some $n \in \Z^+$ and some torsion group $T$.  By properties (1), (2), and (4) of Definition~\ref{size-function-def}, $F(G) = F(\Z^n \oplus T) = F(\Z^n) + F(T) = F(\Z^n) = n F(\Z) = (\rank G) F(\Z) = k \rank G$.
\end{proof}

\begin{lemma} \label{exact-size-functions-subgroups-quotients-zero-lem}
If $F : \Ab \to \Z^+ \cup \{ \infty \}$ is an exact size function, $G$ is an abelian group with $F(G) = 0$, and $H$ is a subgroup of $G$, then $F(H) = 0$ and $F(G/H) = 0$.
\end{lemma}

\begin{proof}
Since the sequence $0 \to H \to G \to G/H \to 0$ is exact, we have $F(G) = F(H) + F(G/H)$.  However, since $F(G) = 0$ and the values of $F(H)$ and $F(G/H)$ are non-negative, we must have $F(H) = 0$ and $F(G/H) = 0$.
\end{proof}

\begin{example}[Examples of Exact Size Functions]\label{example exact size functions}
Parts (i), (iii), and (iv) of Proposition~\ref{rank-elem-facts-prop} show that $\rank$ is an exact size function on the class of abelian groups.  Furthermore, in analogy with \eqref{rank-def-eqn}, one can generalize this function as follows:  If $\mathsf{k}$ is any field of characteristic $0$, then $\mathsf{k}$ may be viewed as a $\Z$-module and for any abelian group $G$, the tensor product $\mathsf{k} \otimes_\Z G$ is a vector space over $\mathsf{k}$.  We may then define
$$\rank_\mathsf{k} (G) := \dim_\mathsf{k} (\mathsf{k} \otimes_\Z G)$$
where $\dim_\mathsf{k}$ denotes the dimension as a $\mathsf{k}$-vector space.  

It is straightforward to verify properties (1) and (2) of Definition~\ref{exact-size-function-def}, and property (3) of Definition~\ref{exact-size-function-def} follows from the fact $\mathsf{k}$ is flat as a $\Z$-module.  (Recall that a $\Z$-module is flat if and only if it is torsion free.)  Thus $\rank_\mathsf{k}$ is an exact size function, and when $\mathsf{k} = \Q$ we recover the usual $\rank$ function.

Another example of an exact size function is
$$
F(G) = \begin{cases}
\infty &\text{if $G$ is not a torsion group} \\
0 &\text{if $G$ is a torsion group}.
\end{cases}
$$
In particular, $F( \Z ) = \infty$.
\end{example}

\begin{example}[Examples of Size Functions] \label{size-fct-not-exact-ex}
Parts (i) and (iii) of Proposition~\ref{corank-elem-facts-prop}, Proposition~\ref{quotient-zero-no-change-prop}, and Proposition~\ref{corank-direct-sum-prop} show that $\corank$ is a size function on the class of abelian groups, and Example~\ref{corank-not-exact-ex} shows that $\corank$ is not an exact size function.  If $X$ is torsion-free, then $F(G) : = \rank ( \operatorname{Hom} (G,X))$ is a size function.  We have also that $F(G) : = \rank ( \operatorname{Hom} (G,X))$ is an exact size function if $X$ is torsion-free and divisible.  However, if $X$ is an abelian group, then $F(G) := \rank ( \operatorname{Hom} (G,X))$ need not be a size function in general.  Moreover, if we take $F(G) := \rank ( \operatorname{Hom} (G,\Z))$, then $F(G) = \corank (G)$. 
\end{example}

\section{Size Functions and $K$-theory of unital Leavitt path algebras}

\subsection{Using exact size functions to determine the number of singular vertices}

\begin{theorem} \label{exact-size-fct-singular-from-K-groups-thm}
Suppose $E$ is a graph with finitely many vertices, $\mathsf{k}$ is a field, and $F : \Ab \to \Z^+ \cup \{ \infty \}$ is an exact size function (see Definition~\ref{exact-size-function-def}).  If $n \in \N$ is a natural number for which $F(K_n(\mathsf{k})) < \infty$, and $0 < F(K_{n-1}(\mathsf{k})) < \infty$, then $F(K_n (L_\mathsf{k}(E)) ) < \infty$ and
$$|E^0_\textnormal{sing}| = \frac{\left( F(K_n(\mathsf{k})) +  F(K_{n-1} (\mathsf{k})) \right)\rank K_0(L_\mathsf{k}(E)) - F(K_n (L_\mathsf{k}(E)))}{F(K_{n-1}(\mathsf{k}))}.$$
\end{theorem}

\begin{proof}
The long exact sequence of Theorem~\ref{abc} induces the short exact sequence
 \begin{align}
0 \longrightarrow \coker &\left( \left(  \begin{smallmatrix} B_E ^t - I \\ C_E ^t \end{smallmatrix} \right) : K_n (\mathsf{k})^{E^0_\textnormal{reg}} \to K_n (\mathsf{k})^{E^0} \right) \longrightarrow K_n (L_\mathsf{k}(E)) \notag \\
 & \qquad   \longrightarrow  \ker \left( \left(  \begin{smallmatrix} B_E ^t - I \\ C_E ^t \end{smallmatrix} \right) : K_{n-1} (\mathsf{k})^{E^0_\textnormal{reg}} \to K_{n-1} (\mathsf{k})^{E^0} \right) \longrightarrow 0. \label{ses-exact-size-fct-eq}
\end{align} 
Since $F$ is an exact size function, it follows from Proposition~\ref{exact-are-size-fcts-prop} that $F$ is also a size function and satisfies the properties listed in Definition~\ref{size-function-def}.  

Theorem~\ref{finite-vertices-computation} implies that there exist $d_1, \ldots, d_k \in \{2, 3, \ldots \}$ and $m \in \Z^+$ such that
\begin{align*} &\coker \left( \left( \begin{smallmatrix} B_E ^t - I \\ C_E ^t \end{smallmatrix} \right) : K_n (\mathsf{k})^{E^0 _\reg} \to K_n (\mathsf{k})^{E^0} \right) \\ 
& \qquad  \cong \bigslant{K_n (\mathsf{k})}{ \langle d_1 x : x \in K_n (\mathsf{k}) \rangle } \oplus \cdots \oplus \bigslant{K_n (\mathsf{k})}{ \langle d_k x : x \in K_n (\mathsf{k}) \rangle } \oplus K_n(\mathsf{k})^{m+|E^0 _\sing|} \end{align*}
and
\begin{align*} &\ker \left( \left( \begin{smallmatrix} B_E ^t - I \\ C_E ^t \end{smallmatrix} \right) : K_{n-1} (\mathsf{k})^{E^0 _\reg} \to K_{n-1} (\mathsf{k})^{E^0} \right) \\ 
& \qquad \qquad \qquad \qquad \cong K_{n-1} (\mathsf{k})^m \oplus \left( \bigoplus_{i=1} ^k \ker ((d_i):K_{n-1} (\mathsf{k}) \to K_{n-1} (\mathsf{k})) \right) \end{align*}
and furthermore, $m$ satisfies $\rank K_0(L_\mathsf{k}(E)) = m+ |E^0_\textnormal{sing} |.$
We may now use the fact that $F$ breaks up over direct sums to evaluate $F$ on the cokernel and kernel.  Since $K_n (\mathsf{k}) /  \langle d_i x : x \in K_n (\mathsf{k}) \rangle$ is a torsion group for all $1 \leq i \leq k$, the size function $F$ assigns a value of zero to these groups, and since $E^0$ is finite, $K_n(\mathsf{k})^{m+|E^0 _\sing|}$ is a finite direct sum and $F\left(K_n(\mathsf{k})^{m+|E^0 _\sing|}\right) = (m+|E^0 _\sing| ) F(K_n(\mathsf{k})) = (\rank K_0(L_\mathsf{k}(E)))  F(K_n(\mathsf{k}))$.  Thus
\begin{equation} \label{coker-comp-exact-size-eq}
F \left( \coker \left( \left( \begin{smallmatrix} B_E ^t - I \\ C_E ^t \end{smallmatrix} \right) : K_n (\mathsf{k})^{E^0 _\reg} \to K_n (\mathsf{k})^{E^0} \right) \right) =  \left( \rank K_0(L_\mathsf{k}(E)) \right)  F(K_n(\mathsf{k})).
\end{equation}
In addition, since $\ker ((d_i):K_n (\mathsf{k}) \to K_n (\mathsf{k}))$ is a torsion group for all $1 \leq i \leq k$, the size function $F$ assigns a value of zero to these groups, and since $m$ is finite, $F(K_{n-1} (\mathsf{k})^m) = m F(K_{n-1} (\mathsf{k})) = (\rank K_0(L_\mathsf{k}(E)) - |E^0_\textnormal{sing} |) F(K_{n-1} (\mathsf{k}))$.  Thus
\smallskip
\begin{equation}  \label{ker-comp-exact-size-eq}
\scalebox{.85}{$
F \left( \ker \left( \left( \begin{smallmatrix} B_E ^t - I \\ C_E ^t \end{smallmatrix} \right) : K_{n-1} (\mathsf{k})^{E^0 _\reg} \to K_{n-1} (\mathsf{k})^{E^0} \right) \right) = (\rank K_0(L_\mathsf{k}(E)) - |E^0_\textnormal{sing} |) F(K_{n-1} (\mathsf{k})).$}
\end{equation}
\smallskip
Since $F$ is an exact size function, we may use \eqref{ses-exact-size-fct-eq}, together with \eqref{coker-comp-exact-size-eq} and \eqref{ker-comp-exact-size-eq}, to deduce
\begin{align*}
F(K_n &(L_\mathsf{k}(E))) \\
&=  \left( \rank K_0(L_\mathsf{k}(E)) \right)  
F(K_n(\mathsf{k})) + (\rank K_0(L_\mathsf{k}(E)) - |E^0_\textnormal{sing} |) F(K_{n-1} (\mathsf{k})) \\
&=  \left( F(K_n(\mathsf{k})) +  F(K_{n-1} (\mathsf{k})) \right) \rank K_0(L_\mathsf{k}(E)) - |E^0_\textnormal{sing} | F(K_{n-1} (\mathsf{k})).
\end{align*}
Since $F(K_n(\mathsf{k})) < \infty$ and $F(K_{n-1}(\mathsf{k})) < \infty$ by hypothesis, and since $\rank K_{0} ( L_{ \mathsf{k}} (E) ) < \infty$, we have that $F(K_n (L_\mathsf{k}(E))) < \infty$.  Also, we obtain   
$$F(K_n (L_\mathsf{k}(E))) - \left( F(K_n(\mathsf{k})) +  F(K_{n-1} (\mathsf{k})) \right) \rank K_0(L_\mathsf{k}(E)) = - |E^0_\textnormal{sing} | F(K_{n-1} (\mathsf{k}))$$
and since $F(K_{n-1}(\mathsf{k})) > 0$ by hypothesis, we may divide to obtain
$$|E^0_\textnormal{sing}| = \frac{ \left( F(K_n(\mathsf{k})) +  F(K_{n-1} (\mathsf{k})) \right) \rank K_0(L_\mathsf{k}(E)) - F(K_n (L_\mathsf{k}(E)))}{F(K_{n-1}(\mathsf{k}))}.$$
\end{proof}

\begin{corollary} \label{K0-Kn-1-Kn-exact-Mor-Eq-cor}
Suppose $E$ and $F$ are simple graphs with finitely many vertices and infinitely many edges, and suppose that $\mathsf{k}$ is a field.  If there exist an exact  size function $F : \Ab \to \Z^+ \cup \{ \infty \}$ and a natural number $n \in \N$ for which $F(K_n(\mathsf{k})) < \infty$, and $0 < F(K_{n-1}(\mathsf{k})) < \infty$, then 
the following are equivalent:
\begin{itemize}
\item[(i)] $L_\textsf{k}(E)$ and $L_\textsf{k}(F)$ are Morita equivalent.
\item[(ii)] $K_0(L_\textsf{k}(E)) \cong K_0(L_\textsf{k}(F))$ and $K_n(L_\textsf{k}(E)) \cong K_n(L_\textsf{k}(F))$.
\end{itemize}
\end{corollary}

\begin{proof}
We obtain $(i) \implies (ii)$ from the fact that Morita equivalent algebras have  algebraic $K$-theory groups that are isomorphic.  For $(ii) \implies (i)$, we see that $(ii)$ combined with Theorem~\ref{exact-size-fct-singular-from-K-groups-thm} implies that $|E^0_\textnormal{sing} | = |F^0_\textnormal{sing}|$ and $K_0(L_\textsf{k}(E)) \cong K_0(L_\textsf{k}(F))$.  It follows from \cite[Theorem~7.4]{rt} that $L_\textsf{k}(E)$ and $L_\textsf{k}(F)$ are Morita equivalent.
\end{proof}

\begin{remark}
Since $\rank$ is an exact size function, both Theorem~\ref{exact-size-fct-singular-from-K-groups-thm} and Corollary~\ref{K0-Kn-1-Kn-exact-Mor-Eq-cor} apply when $F(-) = \rank(-)$.
\end{remark}

\begin{remark}
Note that in both Theorem~\ref{exact-size-fct-singular-from-K-groups-thm} and Corollary~\ref{K0-Kn-1-Kn-exact-Mor-Eq-cor} the value of $n=1$ is allowed.  Also note that since $K_0(L_\textsf{k}(F))$ is a finitely generated abelian group, we always have $F(K_0(L_\textsf{k}(F))) = F( \Z ) \rank K_0(L_\textsf{k}(F))$ for any size function $F$, by Proposition~\ref{constant-multiple-rank-prop}.
\end{remark}

\subsection{Using size functions to determine the number of singular vertices}

\begin{theorem} \label{size-fct-singular-from-K-groups-thm}
Suppose $E$ is a graph with finitely many vertices, $\mathsf{k}$ is a field, and $F : \Ab \to \Z^+ \cup \{ \infty \}$ is a size function (see Definition~\ref{size-function-def}).  If $n \in \N$ is a natural number for which $F(K_n(\mathsf{k})) = 0$, and $0 < F(K_{n-1}(\mathsf{k})) < \infty$, then  $F(K_n (L_\mathsf{k}(E)) ) < \infty$ and
$$|E^0_\textnormal{sing}| = \rank K_0(L_\mathsf{k}(E)) - \frac{F(K_{n} (L_\mathsf{k}(E)))}{F(K_{n-1} (\mathsf{k}))}.$$
\end{theorem}

\begin{proof}
The long exact sequence of Theorem~\ref{abc} induces the short exact sequence
 \begin{align}
0 \longrightarrow \coker &\left( \left(  \begin{smallmatrix} B_E ^t - I \\ C_E ^t \end{smallmatrix} \right) : K_n (\mathsf{k})^{E^0_\textnormal{reg}} \to K_n (\mathsf{k})^{E^0} \right) \longrightarrow K_n (L_\mathsf{k}(E)) \notag \\
 & \qquad   \longrightarrow  \ker \left( \left(  \begin{smallmatrix} B_E ^t - I \\ C_E ^t \end{smallmatrix} \right) : K_{n-1} (\mathsf{k})^{E^0_\textnormal{reg}} \to K_{n-1} (\mathsf{k})^{E^0} \right) \longrightarrow 0 \label{ses-general-size-fct-eq}
\end{align} 
and Theorem~\ref{finite-vertices-computation} implies that there exist $d_1, \ldots, d_k \in \{2, 3, \ldots \}$ and $m \in \Z^+$ such that
\begin{align*} &\coker \left( \left( \begin{smallmatrix} B_E ^t - I \\ C_E ^t \end{smallmatrix} \right) : K_n (\mathsf{k})^{E^0 _\reg} \to K_n (\mathsf{k})^{E^0} \right) \\ 
& \qquad  \cong \bigslant{K_n (\mathsf{k})}{ \langle d_1 x : x \in K_n (\mathsf{k}) \rangle } \oplus \cdots \oplus \bigslant{K_n (\mathsf{k})}{ \langle d_k x : x \in K_n (\mathsf{k}) \rangle } \oplus K_n(\mathsf{k})^{m+|E^0 _\sing|} \end{align*}
and
\begin{align*} &\ker \left( \left( \begin{smallmatrix} B_E ^t - I \\ C_E ^t \end{smallmatrix} \right) : K_{n-1} (\mathsf{k})^{E^0 _\reg} \to K_{n-1} (\mathsf{k})^{E^0} \right) \\ 
& \qquad \qquad \qquad \qquad \cong K_{n-1} (\mathsf{k})^m \oplus \left( \bigoplus_{i=1} ^k \ker ((d_i):K_{n-1} (\mathsf{k}) \to K_{n-1} (\mathsf{k})) \right) \end{align*}
and furthermore, $m$ satisfies $$\rank K_0(L_\mathsf{k}(E)) = m+ |E^0_\textnormal{sing} |.$$
We may now use the fact that $F$ breaks up over direct sums to evaluate $F$ on the cokernel and kernel.  Since $K_n (\mathsf{k}) /  \langle d_i x : x \in K_n (\mathsf{k}) \rangle$ is a torsion group for all $1 \leq i \leq k$, the size function $F$ assigns a value of zero to these groups, and since $E^0$ is finite and $F(K_n(\mathsf{k})) = 0$ by hypothesis, we may conclude that $F \left( K_n(\mathsf{k})^{m+|E^0 _\sing|} \right) = 0$.  Thus
\begin{equation} \label{coker-comp-general-size-eq}
F \left( \coker \left( \left( \begin{smallmatrix} B_E ^t - I \\ C_E ^t \end{smallmatrix} \right) : K_n (\mathsf{k})^{E^0 _\reg} \to K_n (\mathsf{k})^{E^0} \right) \right) =  0.
\end{equation}
In addition, since $\ker ((d_i):K_n (\mathsf{k}) \to K_n (\mathsf{k}))$ is a torsion group for all $1 \leq i \leq k$, the size function $F$ assigns a value of zero to these groups, and since $m$ is finite, $F(K_{n-1} (\mathsf{k})^m) = m F(K_{n-1} (\mathsf{k})) = (\rank K_0(L_\mathsf{k}(E)) - |E^0_\textnormal{sing} |) F(K_{n-1} (\mathsf{k}))$.  Thus
\smallskip
\begin{equation}  \label{ker-comp-general-size-eq}
\scalebox{.85}{$
F \left( \ker \left( \left( \begin{smallmatrix} B_E ^t - I \\ C_E ^t \end{smallmatrix} \right) : K_{n-1} (\mathsf{k})^{E^0 _\reg} \to K_{n-1} (\mathsf{k})^{E^0} \right) \right) = (\rank K_0(L_\mathsf{k}(E)) - |E^0_\textnormal{sing} |) F(K_{n-1} (\mathsf{k})).$}
\end{equation}
\smallskip
Using the short exact sequence in \eqref{ses-general-size-fct-eq}, Lemma~\ref{size-fct-limited-exactness-lem}, and the computations in \eqref{coker-comp-general-size-eq} and \eqref{ker-comp-general-size-eq}, we obtain $F(K_n (L_\mathsf{k}(E))) = (\rank K_0(L_\mathsf{k}(E)) - |E^0_\textnormal{sing} |) F(K_{n-1} (\mathsf{k}))$.  Moreover, this equation together with the hypothesis that $F(K_{n-1}(\mathsf{k})) < \infty$ and the fact that $\rank K_0(L_\mathsf{k}(E)) < \infty$ implies $F(K_n (L_\mathsf{k}(E))) < \infty$.  In addition, since $0 < F(K_{n-1}(\mathsf{k})) < \infty$ by hypothesis and since $\rank K_{0} ( L_{ \mathsf{k}} (E) ) < \infty$, we may divide to obtain
$$\frac{F(K_n (L_\mathsf{k}(E)))}{F(K_{n-1} (\mathsf{k}))} = \rank K_0(L_\mathsf{k}(E)) - |E^0_\textnormal{sing}|,$$
and $|E^0_\textnormal{sing}| = \rank K_0(L_\mathsf{k}(E)) - \frac{F(K_{n} (L_\mathsf{k}(E)))}{F(K_{n-1} (\mathsf{k}))}.$
\end{proof}

\begin{remark}
Although Theorem~\ref{exact-size-fct-singular-from-K-groups-thm} and Theorem~\ref{size-fct-singular-from-K-groups-thm} are similar, neither implies the other.  The hypotheses of Theorem~\ref{exact-size-fct-singular-from-K-groups-thm} require $F$ to be an exact size function, while Theorem~\ref{size-fct-singular-from-K-groups-thm} allows $F$ to be any size function.  Furthermore, the hypotheses of Theorem~\ref{size-fct-singular-from-K-groups-thm} require $F(K_n(\mathsf{k})) = 0$, while Theorem~\ref{exact-size-fct-singular-from-K-groups-thm} only requires $F(K_n(\mathsf{k}))$ to be finite.  Thus the hypotheses of Theorem~\ref{exact-size-fct-singular-from-K-groups-thm} impose stronger conditions on the properties of $F$, while the hypotheses of Theorem~\ref{size-fct-singular-from-K-groups-thm} impose stronger conditions on the value $F(K_n(\mathsf{k}))$.
\end{remark}

\begin{corollary} \label{K0-Kn-1-Kn-size-fct-Mor-Eq-cor}
Suppose $E$ and $F$ are simple graphs with finitely many vertices and an infinite number of edges, and suppose that $\mathsf{k}$ is a field.  If there exist a size function $F : \Ab \to \Z^+ \cup \{ \infty \}$ and a natural number $n \in \N$ for which $F(K_n(\mathsf{k})) =0$, and $0 < F(K_{n-1}(\mathsf{k})) < \infty$, then 
the following are equivalent:
\begin{itemize}
\item[(i)] $L_\textsf{k}(E)$ and $L_\textsf{k}(F)$ are Morita equivalent.
\item[(ii)] $K_0(L_\textsf{k}(E)) \cong K_0(L_\textsf{k}(F))$ and $K_n(L_\textsf{k}(E)) \cong K_n(L_\textsf{k}(F))$.
\end{itemize}
\end{corollary}

\begin{proof}
We obtain $(i) \implies (ii)$ from the fact that Morita equivalent algebras have  algebraic $K$-theory groups that are isomorphic.  For $(ii) \implies (i)$, we see that $(ii)$ combined with Theorem~\ref{size-fct-singular-from-K-groups-thm} implies that $|E^0_\textnormal{sing} | = |F^0_\textnormal{sing}|$ and $K_0(L_\textsf{k}(E)) \cong K_0(L_\textsf{k}(F))$.  It follows from \cite[Theorem~7.4]{rt} that $L_\textsf{k}(E)$ and $L_\textsf{k}(F)$ are Morita equivalent.
\end{proof}

\begin{remark}
Since $\corank$ is a size function, Theorem~\ref{size-fct-singular-from-K-groups-thm} and Corollary~\ref{K0-Kn-1-Kn-size-fct-Mor-Eq-cor} apply when $F(-) = \corank(-)$. 
\end{remark}

\begin{remark}
Note that in both Theorem~\ref{size-fct-singular-from-K-groups-thm} and Corollary~\ref{K0-Kn-1-Kn-size-fct-Mor-Eq-cor} the value of $n=1$ is allowed.  Also note that since $K_0(L_\textsf{k}(F) ) $ is a finitely generated abelian group, we always have $F(K_0(L_\textsf{k}(F))) = F( \Z ) \rank K_0(L_\textsf{k}(F)) $ for any size function $F$, by Proposition~\ref{constant-multiple-rank-prop}.
\end{remark}

\begin{remark}
It was proven in \cite[Theorem~7.4]{rt} that if $E$ is a simple graph with a finite number of vertices and an infinite number of edges, then $(K_0(L_\textsf{k}(E)), |E^0_\textnormal{sing}|)$ is a complete Morita equivalence invariant for $L_\mathsf{k}(E)$.  Moreover, it was proven in \cite[Corollary~6.14]{rt} that if $\mathsf{k}$ is a field with no free quotients, then $|E^0_\textnormal{sing}|$ is determined by the pair $(K_0(L_\textsf{k}(E)), K_1(L_\textsf{k}(E)))$, and hence $(K_0(L_\textsf{k}(E)), K_1(L_\textsf{k}(E)))$ is a complete Morita equivalence invariant for $L_\mathsf{k}(E)$ in this case.  Since a field $\mathsf{k}$ has no free quotients if and only if the abelian group $K_1(\mathsf{k}) \cong \mathsf{k}^\times$ has no free quotients, we see that $\mathsf{k}$ has no free quotients if and only if $\corank K_1(\mathsf{k}) =0$.  Thus the result from \cite{rt} is a special case of Corollary~\ref{K0-Kn-1-Kn-size-fct-Mor-Eq-cor} when $n=1$ and $F(-) = \corank(-)$.
\end{remark}

\subsection{Number Fields}
A \emph{number field} is a finite field extension of $\Q$. (We note that, in particular, $\Q$ itself is considered a number field.)  If $\mathsf{k}$ is a number field of degree $n$ over $\Q$, then by the primitive element theorem we may write $\mathsf{k} = \Q (\alpha)$ for an element $\alpha$ of degree $n$.   If we let $p(x)$ be the minimal polynomial of $\alpha$, then since $\Q$ has characteristic zero, $p(x)$ is separable and we may factor the polynomial $p(x)$ into $n$ monomials with distinct roots.  These roots will appear as distinct real numbers together with distinct conjugate pairs, and we write 
$$p(x)=(x-\lambda_1) \ldots (x-\lambda_{r_1}) (x-\mu_1) (x-\overline{\mu}_1) \ldots (x-\mu_{r_2}) (x - \overline{\mu}_{r_2})$$
for distinct elements $\lambda_1, \ldots, \lambda_{r_1} \in \R$ and $\mu_1, \ldots, \mu_{r_2} \in \C \setminus \R$.  Moreover, if we let $q_i(x) := (x-\mu_i) (x-\overline{\mu}_i)$ for $1 \leq i \leq r_2$ be the degree 2 polynomial in $\R[x]$ with $\mu_i$ and $\overline{\mu}_i$ as roots, then $p(x) = (x-\lambda_1) \ldots (x-\lambda_{r_1}) q_1(x) \ldots q_{r_2} (x)$ is a factorization of $p(x)$ into irreducible factors over $\R$.  If we tensor $\mathsf{k}$ with $\R$, we may use the Chinese remainder theorem to obtain
\begin{align*}
\mathsf{k} &\otimes_\Q \R \\
&\cong \Q (\alpha) \otimes_\Q \R \\
&\cong (\Q [x] / \langle p(x) \rangle) \otimes_\Q \R \\
&\cong \R [x] / \langle p(x) \rangle \\
&\cong \R [x] / \langle (x-\lambda_1) \ldots (x-\lambda_{r_1}) q_1(x) \ldots q_{r_2}(x) \rangle \\
&\cong \R[x] /  \langle (x-\lambda_1) \rangle \times \ldots \times  \R[x] /  \langle (x-\lambda_{r_1}) \rangle \times 
 \R[x] /  \langle q_1(x) \rangle \times \ldots \times  \R[x] /  \langle q_{r_2}(x) \rangle \\
&\cong \R (\lambda_1) \times \ldots \times \R( \lambda_{r_1}) \times \R ( \mu_1) \times \ldots \times R(\mu_{r_2}) \\
&\cong \R ^{r_1} \times \C ^{r_2}.
\end{align*}
Since $r_1$ is the number of real roots of $p(x)$ and $r_2$ is the number of conjugate pairs of non-real roots of $p(x)$, we have that $r_1, r_2 \in \Z^+$ and $r_1 + 2r_2 = n$.  We observe (and this will be useful for us later) that at least one of $r_1$ and $r_2$ is strictly positive.   The non-negative integer $r_1$ is called \emph{the number of real places of $\mathsf{k}$}, and the non-negative integer $r_2$ is called \emph{the number of complex places of $\mathsf{k}$}.  If $r_2 = 0$, then $\mathsf{k}$ is said to be \emph{totally real}, and if $r_1 = 0$, then $\mathsf{k}$ is said to be \emph{totally complex}.

\begin{theorem}[Theorem 1.5 of \cite{grayson} or Theorem IV.1.18 of \cite{kbook}] \label{number-field-rank}
Let $\mathsf{k}$ be a number field with $r_1$ real places and $r_2$ complex places. Then for $n \in \Z^+$,
$$\rank K_n(\mathsf{k})= \left\{ \begin{array}{ll} 1 & n=0 \\ 
\infty & n=1 \\
0 & n=2k \text{ and } k>0 \\
r_1 + r_2 & n=4k+1 \text{ and } k >0 \\
r_2 & n=4k+3 \text{ and } k \geq 0. \end{array} \right.$$
\end{theorem}
\noindent In particular, $K_{6+4k} (\mathsf{k})$ is a torsion group and $K_{5+4k} (\mathsf{k})$ has strictly positive finite rank for any $k \in \Z^{+}$.

\begin{theorem}\label{number-field-sing-thm}
Let $\mathsf{k}$ be a number field with $r_1$ real places and $r_2$ complex places. 
If $E$ is a graph with finitely many vertices, then for any $k \in \Z^+$
$$|E^0 _\sing|=\rank(K_0 (L_\mathsf{k} (E)))- \frac{\rank(K_{6+4k} (L_\mathsf{k} (E)))}{r_1 + r_2}.$$
In addition, if $E$ and $F$ are simple graphs with finitely many vertices and an infinite number of edges, then the following are equivalent:
\begin{itemize}
\item[(i)] $L_\textsf{k}(E)$ and $L_\textsf{k}(F)$ are Morita equivalent.
\item[(ii)] $K_0(L_\textsf{k}(E)) \cong K_0(L_\textsf{k}(F))$ and $K_{6+4k}(L_\textsf{k}(E)) \cong K_{6+4k}(L_\textsf{k}(F))$ for all $k \in \Z^+$.
\item[(iii)] $K_0(L_\textsf{k}(E)) \cong K_0(L_\textsf{k}(F))$ and $K_{6+4k}(L_\textsf{k}(E)) \cong K_{6+4k}(L_\textsf{k}(F))$ for some $k \in \Z^+$.
\end{itemize}
In addition, if  $\mathsf{k}$ is not totally real (i.e., $r_2 \neq 0$), then whenever $E$ is a graph with finitely many vertices and $k \in \Z^+$ we have
$$|E^0 _\sing|=\rank(K_0 (L_\mathsf{k} (E)))- \frac{\rank(K_{4+4k} (L_\mathsf{k} (E)))}{r_2}.$$
Moreover, if $\mathsf{k}$ is not totally real (i.e., $r_2 \neq 0$), then whenever $E$ and $F$ are simple graphs with finitely many vertices and an infinite number of edges, the following are equivalent:
\begin{itemize}
\item[(i)] $L_\textsf{k}(E)$ and $L_\textsf{k}(F)$ are Morita equivalent.
\item[(ii)] $K_0(L_\textsf{k}(E)) \cong K_0(L_\textsf{k}(F))$ and $K_{4+2k}(L_\textsf{k}(E)) \cong K_{4+2k}(L_\textsf{k}(F))$ for any $k \in \Z^+$.
\item[(iii)] $K_0(L_\textsf{k}(E)) \cong K_0(L_\textsf{k}(F))$ and $K_{4+2k}(L_\textsf{k}(E)) \cong K_{4+2k}(L_\textsf{k}(F))$ for some $k \in \Z^+$.
\end{itemize}
\end{theorem}

\begin{proof}
Proposition~\ref{rank-elem-facts-prop} and Proposition~\ref{exact-are-size-fcts-prop} imply that $\rank(-)$ is a size function.  We now apply Theorem~\ref{size-fct-singular-from-K-groups-thm} and Corollary~\ref{K0-Kn-1-Kn-size-fct-Mor-Eq-cor} noting that Theorem~\ref{number-field-rank} implies that $\rank K_{2k}(\mathsf{k}) = 0$,  $\rank K_{3+4k}(\mathsf{k}) = r_2$, and $\rank K_{5+4k}(\mathsf{k}) = r_1 + r_2 $.
\end{proof}

\begin{remark}
Let $E$ be a simple graph with finitely many vertices and an infinite number of edges.  Theorem~\ref{number-field-sing-thm} shows that if $\mathsf{k}$ is a number field, then the pair $(K_0(L_\mathsf{k}(E)), K_6(L_\mathsf{k}(E)))$ is a complete Morita equivalence invariant for $L_\mathsf{k}(E)$, and if $\mathsf{k}$ is not totally real, then the pair $(K_0(L_\mathsf{k}(E)), K_4(L_\mathsf{k}(E)))$ is a complete Morita equivalence invariant for $L_\mathsf{k}(E)$.  
\end{remark}


\begin{thebibliography}{99}

\bibitem{alps}
G.~Abrams, A.~Louly, E.~Pardo, and C.~Smith,
\emph{Flow invariants in the classification of Leavitt path algebras}, J. Algebra. \textbf{333} (2011), 202--231.

\bibitem{ap:JA}
G.~Abrams and G.~Aranda Pino,
\emph{The Leavitt path algebra of a graph}, J. Algebra. \textbf{293} (2005), 319--334. 

\bibitem{ap}
G.~Abrams and G.~Aranda Pino,
\emph{The Leavitt path algebras of arbitrary graphs}, Houston J. of Mathematics. \textbf{32} (2008), 423--442. 


\bibitem{abc}
P.~Ara, M.~Brustenga, and G.~Corti{\~n}as,
\emph{{$K$}-theory of {L}eavitt path algebras}, M\"unster Journal of Mathematics. \textbf{2} (2009), 5--33.

\bibitem{amp}
P.~Ara, M.A.~Moreno, E.~Pardo, \emph{Nonstable {$K$}-theory for {L}eavitt path algebras}, Alg. Rep. Theory \textbf{10} (2007), 157--178.

\bibitem{dt2}
D.~Drinen and M.~Tomforde,
\emph{Computing {$K$}-theory and Ext for graph $C^*$-algebras}, Illinois J. Math. \textbf{46} (2002), 81--91.

\bibitem{dt}
D.~Drinen and M.~Tomforde,
\emph{The $C^*$-algebras of arbitrary graphs}, Rocky Mountain J. Math. \textbf{35} (2005),105--135.

\bibitem{Fuc}
L.~ Fuchs, Infinite Abelian Groups. Vol. II. Pure and Applied Mathematics. Vol. 36-II. Academic Press, New York-London, 1973. ix+363 pp.

\bibitem{Gallian}
J.~Gallian, Contemporary Abstract Algebra, Seventh Edition, Cengage Learning, 2012, xii+656~pp.

\bibitem{grayson}
D.~Grayson. \emph{On the {$K$}-theory of fields}, Contemp. Math. \textbf{83} (1989), 31--55.

\bibitem{Hungerford}
T.~Hungerford. Algebra.  Graduate Texts in Mathematics, \textbf{73}. Springer-Verlag, New York-Berlin, 1980, xxiii+502~pp.	

\bibitem{Kir-pre}
E. Kirchberg, \emph{The classification of purely infinite $C^*$-algebras using Kasparov's theory}, preprint.

\bibitem{Lang}
S.~Lang, Algebra. Revised third edition. Graduate Texts in Mathematics, \textbf{211}. Springer-Verlag, New York, 2002. xvi+914~pp.

\bibitem{Phi}
N.~C.~Phillips, \emph{A classification theorem for nuclear purely infinite simple $C^*$-algebras}, Doc. Math. \textbf{5} (2000), 49--114.

\bibitem{quillen}
D.~Quillen,
\emph{On the {C}ohomology and {$K$}-theory of the general linear groups over a finite field},
Annals of Mathematics, Second Series. \textbf{96} (1972), 552--586.

\bibitem{rosen}
J.~Rosenberg. \emph{Algebraic {$K$}-theory and its applications}. Graduate Texts in Mathematics, \textbf{147}. Springer-Verlag, New York, 1994. x+392 pp.

\bibitem{rt}
E.~Ruiz and M.~Tomforde,
\emph{Classification of unital simple Leavitt path algebras of infinite graphs}, J. Algebra. \textbf{384} (2013), 45--83.

\bibitem{rtideal}
E.~Ruiz and M.~Tomforde,
\emph{Ideal-related {$K$}-theory for Leavitt path algebras and graph $C^*$-algebras}, Indiana Univ.~Math.~J. \textbf{62} (2013), no. 5, 1587--1620. 


\bibitem{sorensen}
A.~S{\o}rensen,
\emph{Geometric classification of simple graph algebras}, Ergodic Theory and Dynamical Systems. \textbf{33} (2013), 1199--1220.

\bibitem{kbook}
C.~Weibel,
\emph{The {$K$}-book: an introduction to algebraic {$K$}-theory}. AMS, Rhode Island, 2013. xii+618 pp.

\end{thebibliography}
\end{document}